\documentclass[1 [leqno,12pt]{amsart}
\usepackage{txfonts, color}
\usepackage{mathrsfs}
\usepackage{amsmath}
\usepackage{amssymb, amsmath}
\usepackage{graphicx}

\newcommand{\newcom}{\newcommand}

\newcom{\al}{\alpha}
\newcom{\be}{\beta}
\newcom{\eps}{\epsilon}
\newcom{\veps}{\varepsilon}
\newcom{\e}{\varepsilon}
\newcom{\ga}{\gamma}
\newcom{\Ga}{\Gamma}
\newcom{\ka}{\kappa}
\newcom{\Lam}{\Lambda}
\newcom{\lam}{\lambda}
\newcom{\Om}{\Omega}
\newcom{\om}{\omega}
\newcom{\Si}{\Sigma}
\newcom{\si}{\sigma}
\newcom{\tht}{\theta}
\newcom{\dtri}{\nabla}
\newcom{\tri}{\triangle}
\newcom{\oo}{\infty}
\newcom{\vphi}{\varphi}
\newcom{\cB}{{\mathcal B}}
\newcom{\cC}{{\mathcal C}}
\newcom{\cD}{{\mathcal D}}
\newcom{\cF}{{\mathcal F}}
\newcom{\cH}{{\mathcal H}}
\newcom{\cL}{{\mathcal L}}
\newcom{\cM}{{\mathcal M}}
\newcom{\cN}{{\mathcal N}}
\newcom{\cP}{{\mathcal P}}
\newcom{\cS}{{\mathcal S}}
\newcom{\cQ}{{\mathcal Q}}
\newcom{\cT}{{\mathcal T}}
\newcom{\cY}{{\mathcal Y}}
\newcom{\cZ}{{\mathcal Z}}
\newcom{\R}{\mathbb R}
\newcom{\T}{\mathbb T}
\newcom{\BT}{{\mathbb{T}^2}}
\newcom{\Z}{\mathbb Z}
\newcom{\C}{\mathbb C}
\newcom{\E}{\mathbb E}

\newcommand{\vc}[1]{{\bf #1}}
\newcom{\ve}{\vc{e}}
\newcom{\vN}{\vc{N}}
\newcom{\vn}{\vc{n}}
\newcom{\vG}{\vc{G}}
\newcom{\vF}{\vc{F}}
\newcom{\vf}{\vc{f}}
\newcom{\vg}{\vc{g}}
\newcom{\vq}{\vc{q}}
\newcom{\vu}{\vc{u}}
\newcom{\vv}{\vc{v}}
\newcom{\vw}{\vc{w}}
\newcom{\vb}{\vc{b}}
\newcom{\vh}{\vc{h}}
\newcom{\vz}{\vc{z}}
\newcom{\vup}{\vu^{+}}
\newcom{\vum}{\vu^{-}}
\newcom{\vvp}{\vv^{+}}
\newcom{\vvm}{\vv^{-}}
\newcom{\vbp}{\vb^{+}}
\newcom{\vbm}{\vb^{-}}
\newcom{\vhp}{\vh^{+}}
\newcom{\vhm}{\vh^{-}}
\newcom{\Omp}{{\Om^+}}
\newcom{\Omm}{{\Om^-}}
\newcom{\vupm}{{\vu^{\pm}}}
\newcom{\vvpm}{{\vv^{\pm}}}
\newcom{\vbpm}{{\vb^{\pm}}}
\newcom{\vhpm}{{\vh^{\pm}}}
\newcom{\vwp}{{\vc{w}^+}}
\newcom{\vwm}{{\vc{w}^-}}
\newcom{\vwpm}{{\vc{w}^{\pm}}}
\newcom{\Ompm}{{\Omega^{\pm}}}
\newcom{\vom}{\boldsymbol{\omega}}
\newcom{\vvap}{\boldsymbol{\varpi}}
\newcom{\vop}{\vom^{+}}
\newcom{\vnu}{\boldsymbol{\nu}}
\newcom{\vopm}{\vom^{\pm}}
\newcom{\vjp}{\vj^+}
\newcom{\vjm}{\vj^-}
\newcom{\vjpm}{\vj^{\pm}}
\newcom{\vj}{\boldsymbol{\xi}}
\newcom{\Ds} {\langle\nabla\rangle^{s-\f12}}

\newcommand{\dx}{{\rm d} {x}}

\newcom{\ds}{{\rm d} s}

\newcom{\f}{\frac}
\newcom{\di}{\displaystyle\int}
\newcom{\dl}{\displaystyle\lim}
\newcom{\ov}{\overline}
\newcom{\sset}{\subset}
\newcom{\wt}{\widetilde}
\newcom{\pa}{\partial}
\newcom{\p}{\partial}
\newcom\na{\nabla}
\newcom{\suml}{\sum\limits}
\newcom{\supl}{\sup\limits}
\newcom{\intl}{\int\limits}
\newcom{\infl}{\inf\limits}
\newcom{\disp}{\displaystyle}
\newcom{\non}{\nonumber}
\newcom{\no}{\noindent}
\newcom{\QED}{$\square$}

\def\div{\mathop{\rm div}\nolimits}
\def\curl{\mathop{\rm curl}\nolimits}

\def\eqdefa{\buildrel\hbox{\footnotesize def}\over =}

\newtheorem{athm}{\bf \t}[section]
\newenvironment{thm} [1] {\def\t{#1}\begin{athm} \bf \rm} {\end{athm}}
\newcom{\bthm}{\begin{thm}}\newcom{\ethm}{\end{thm}}

\newtheorem{theorem}{Theorem}[section]
\newtheorem{lemma}{Lemma}[section]
\newtheorem{remark}{Remark}[section]

\newtheorem{definition}{Definition}[section]
\newtheorem{proposition}{Proposition}[section]

\newcom{\beq}{\begin{equation}}
\newcom{\eeq}{\end{equation}}
\newcom{\ben}{\begin{eqnarray}}
\newcom{\een}{\end{eqnarray}}
\newcom{\beno}{\begin{eqnarray*}}
\newcom{\eeno}{\end{eqnarray*}}
\newcom{\bali}{\begin{aligned}}
\newcom{\eali}{\end{aligned}}

\numberwithin{equation}{section}

\begin{document}

\title[Nonlinear stability of current-vortex sheet]
{Nonlinear stability of current-vortex sheet to the incompressible MHD equations}

\author{Yongzhong Sun}
\address{Department of Mathematics, Nanjing University, 210093, Nanjing, P. R. China}
\email{sunyz@nju.edu.cn}

\author{Wei Wang}

\address{School of Mathematical Sciences, Zhejiang University, 310027, Hangzhou, P. R. China}
\email{wangw07@zju.edu.cn}

\author{Zhifei Zhang}
\address{School of Mathematical Sciences, Peking University, 100871, P. R. China}
\email{zfzhang@math.pku.edu.cn}

\date{\today}

\maketitle

\begin{abstract}
In this paper, we solve a long-standing open problem: nonlinear stability of current-vortex sheet in the ideal incompressible Magneto-Hydrodynamics under the linear stability condition.
This result gives a first rigorous confirmation of the stabilizing effect of the magnetic field on Kelvin-Helmholtz instability.
\end{abstract}

\section{Introduction}\label{prob}

\subsection{Presentation of the problem}

In this paper, we consider the idea incompressible
Magneto-Hydrodynamics(MHD). Let $\vu=(u^1, u^2, u^3)$ be the velocity
and $\vh=(h^1, h^2, h^3)$ be the magnetic field.
The incompressible MHD system reads as follows
\beq\label{mhd}\left\{\begin{split}
& \p_t \vu + \vu\cdot\nabla\vu - \vh\cdot\nabla\vh + \nabla p= 0, \\
& \div \vu = 0,\quad \div\vh=0, \\
& \p_t \vh + \vu\cdot\nabla\vh - \vh\cdot\nabla\vu = 0.
\end{split}\right.\eeq

We denote
\[
Q_{T} \triangleq \cup_{t\in (0,T)} \{t\} \times \Om_t\subset (0,\infty)\times \mathbb{R}^3.
\]
Let $\vu,\vh$ be a weak solution of the MHD system (\ref{mhd}). A current-vortex sheet is a moving surface $\Gamma(t)\subset Q_T$ such that for
\[
\Om_t=\Om^+_t\cup \Gamma(t) \cup \Om^-_t, \quad Q^{\pm}_T = \cup_{t\in (0,T)} \{t\} \times \Om^{\pm}_t,
\]
the solution
\[
\vupm := \vu|_{\Om^{\pm}_t}, \quad \vhpm := \vh|_{\Om^{\pm}_t}, \quad p^{\pm} := p|_{\Om^{\pm}_t}
\]
are smooth in $Q^{\pm}_T$ and satisfy
\beq\label{cvs}\left\{\begin{split}
& \p_t \vupm + \vupm\cdot\nabla\vupm - \vhpm\cdot\nabla\vhpm + \nabla p^{\pm} = 0\quad \text{ in }\quad Q^{\pm}_T,  \\
& \div \vupm = 0,\,\div \vhpm = 0\quad\text{ in }\quad Q^{\pm}_T,  \\
& \p_t \vhpm + \vupm\cdot\nabla\vhpm - \vhpm\cdot\nabla\vupm = 0\quad\text{ in }\quad Q^{\pm}_T,  \\
&\vupm\cdot\vn =V(t,x), \quad \vhpm\cdot\vn = 0\quad\text{ on }\quad\Gamma_t,
\end{split}\right.\eeq
with jump condition for the pressure
\begin{align}
[p]\eqdefa p^+-p^-=0\quad \text{on}\quad \Gamma_t.
\end{align}
Here $\vn$ is the outward unit normal to $\p\Om^-_t$ and $V(t,x)$ is the normal velocity of $\Gamma_t$.

The system \eqref{cvs} is  supplemented with initial data
\beq\label{cvs0}
\vupm(0,x) = \vu^\pm_{0}(x), \quad \vhpm(0,x)=\vh^\pm_{0}(x)\quad \text{in}\quad\Om^\pm_0,
\eeq
where the initial data satisfies
\beq\label{cvsi}
\left\{\begin{split}
&\div \vu^\pm_{0}=0, \quad\div\vh^\pm_{0}=0\quad\text{ in}\quad \Om^\pm_0,\\
&\vup_{0}\cdot\vn_0 = \vum_{0}\cdot\vn_0, \quad \vh^\pm_{0}\cdot\vn_0 = 0\quad\text{on}\quad\Gamma_0.
\end{split}\right.
\eeq

The system (\ref{cvs})-(\ref{cvs0}) is so called the current-vortex sheet problem. The goal of this paper is to
study the well-posedness of this system under suitable stability condition on the initial data.

For simplicity, we consider
\[
\Om = \mathbb{T}^2 \times (-1,1), \quad \Gamma_0=\big\{(x',x_3)| x_3 = f_0(x'), \, x'=(x_1,x_2)\in \mathbb{T}^2\big\},
\]
and $\Gamma(t)$ is a graph:
\[
\Gamma_t=\left\{x\in \Om | x_3=f(t,x'), x'=(x_1,x_2)\in \mathbb{T}^2\right\}
\]
such that
\[
\Om^{+}_t=\big\{ x \in \Om(t)| x_3 > f(t,x')\big\}, \quad \Om^{-}_t=\big\{ x \in \Om(t)| x_3 < f(t,x')\big\}.
\]

All functions(and vectors) are assumed to be periodic in $x'$. On the artificial boundary $\Gamma^{\pm}= \mathbb{T}^2 \times \{\pm 1\}$,
we impose the following boundary conditions on $\vupm,\vhpm$:
\beq\label{cvsb2}
u_3^{\pm} =0,\quad h_3^{\pm} = 0\quad\text{on}\quad \Gamma^{\pm}.
\eeq
Under this setting, the boundary condition on $\Gamma(t)$ in (\ref{cvs}) is transformed into
\beq\label{cvsb1}
[p]=0, \quad \vupm\cdot\vN = \p_t f, \quad \vhpm\cdot\vN = 0\quad\text{on}\quad\Gamma_t,
\eeq
where
\[
\vN=(-\p_1f, -\p_2 f, 1), \quad \vn=\frac{\vN}{|\vN|}.
\]

Let us remark that the divergence free restriction on $\vhpm$ is automatically satisfied if $\div\vh^\pm_0=0$, because of
\[
\p_t\div\vhpm + \vupm\cdot\nabla\div\vhpm=0.
\]
Similar argument can be also applied to yield $\vhpm\cdot\vN = 0$ if $\vh^\pm_{0}\cdot\vn_0 = 0$.

\subsection{Backgrounds}

A velocity discontinuity in an inviscid flow is called a vortex sheet. A vortex sheet has vorticity concentrated as a measure(delta function)
in a set of codimension one, a curve or a surface for two dimensional flow or three dimensional flow respectively.
For the 2-D incompressible Euler equations, the evolution of the vortex sheet can be described by Birkhoff-Rott(BR) equation.
The linear stability analysis of BR equation shows that the $k$th Fourier mode of the solution grows like $e^{{|k|t}}$. This instability is so called Kelvin-Helmholtz instability. We refer to \cite{Maj} for more introductions.

In a series of important works \cite{Cou1, Cou2, Cou3}, Coulombel and Secchi proved the nonlinear stability of supersonic compressible vortex sheets for the 2-D isentropic Euler equations. This is a nonlinear hyperbolic equations with free boundary. Moreover, the free boundary is characteristic and the Kreiss-Lopatinskii condition holds only in a weak sense, which yields losses of derivatives. For this, they proved the existence of the solution by using Nash-Moser iteration. On the other hand,
for the 2-D compressible Euler equation with the Mach number $M<\sqrt{2}$ or 3-D compressible Euler equations, the vortex sheet is
violently unstable.

Trakhinin \cite{Tra1} first found a sufficient condition for the neutral stability of planar compressible current-vortex sheet for
a general case of the unperturbed flow. Furthermore, he also proved an {\it a priori} estimate for the linear variable coefficients linearized problem,
which is a key step towards nonlinear problem. Again in this case, the Kreiss-Lopatinskii condition holds only in a weak sense.
The existence of compressible current-vortex sheet was solved independently by Chen-Wang \cite{Chen} and Trakhinin \cite{Tra2} by using Nash-Moser iteration. Recently, Secchi and Trakhinin \cite{Sec} also proved the well-posedness of the plasma-vacuum interface problem in ideal compressible MHD equations. Wang and Yu \cite{WY} analyzed the linear stability of 2-D compressible current-vortex sheet.

The necessary and sufficient condition for the planar(constant coefficients) incompressible
current-vortex sheet was found by Syrovatskii \cite{Sy} and Axford \cite{Ax} for a long time ago.
The linear stability condition reads as follows
\ben
&\big|[\vu]\big|^2\le 2\big(|\vh^+|^2+|\vh^-|^2\big),\label{stabiliy-1}\\
&\big|[\vu]\times\vh^+\big|^2+\big|[\vu]\times\vh^-\big|^2\le 2\big|\vh^+\times\vh^-\big|^2.\label{stabiliy-2}
\een
In particular, if $\vh^+\times\vh^-\neq 0$ and
\begin{align}
\big|[\vu]\times\vh^+\big|^2+\big|[\vu]\times\vh^-\big|^2 < 2\big|\vh^+\times\vh^-\big|^2,\label{stability-weak}
\end{align}
the condition (\ref{stabiliy-1}) is automatically satisfied.  For the current sheet(i.e., $[\vu]=0$ and $\vh^+\times\vh^-\neq 0$),
the condition (\ref{stability-weak}) holds always.

Under the condition (\ref{stability-weak}), Morando, Trakhinin and Trebeschi \cite{Mo1} proved an {\it a priori} estimate with a loss of three derivatives for the linearized system. Under strong stability condition
\ben\label{stability-strong}
\max\Big(\big|[\vu]\times\vh^+\big|, \big|[\vu]\times\vh^-\big|\Big)<\big|\vh^+\times\vh^-\big|,
\een
Trakhinin \cite{Tra-in1} proved an {\it a priori} estimate without loss of derivative from data for the linearized system with variable coefficients.

In a recent work \cite{CMST}, Coulombel, Morando,  Secchi and Trebeschi  proved an {\it a priori estimate} without loss of derivatives for nonlinear current-vortex sheet problem under the strong stability condition (\ref{stability-strong}). This important progress gives some hope for the existence of the solution. However, unlike usual existence theory of the PDE problem,  it is usually highly nontrivial for a free boundary problem  to conclude the existence of the solution from uniform {\it a priori estimates} .

Nonlinear stability of the incompressible current-vortex sheet problem has been an open question, even under the strong stability condition \cite{Tra-r}.
Compared with compressible current-vortex sheet problem, one of main difficulties is that  the incompressible current-vortex sheet problem is not a hyperbolic problem, since the pressure is an unknown determined by
an elliptic equation.

\subsection{Main result}

This paper is devoted to proving nonlinear stability of the system (\ref{cvs})-(\ref{cvsi})
under the weak stability condition
\begin{align*}
\big|[\vu]\times\vh^+\big|^2+\big|[\vu]\times\vh^-\big|^2 < 2\big|\vh^+\times\vh^-\big|^2\quad \text{on}\quad\Gamma_t.
\end{align*}
By Lemma \ref{lem:stability}, weak stability condition implies that
\begin{align}\nonumber
\Lambda(\vh^\pm,[\vu])&\eqdefa\sup_{x\in\Gamma_t}\sup_{\varphi_1^2+\varphi_2^2=1}(h_1^+\varphi_1+h_2^+\varphi_2)^2+(h_1^-\varphi_1+h_2^-\varphi_2)^2
-2(v_1\varphi_1+v_2\varphi_2)^2\\
&\ge c_0
\end{align}
for some $c_0>0$, where $[\vu]=2(v_1,v_2,v_3)$.

\medskip

Now, let us state our main result.

\begin{theorem}
Let $s\ge 3$ be an integer. Assume that
\begin{align*}
f_0\in H^{s+\f12}(\BT),\quad  \vu_0^\pm,\,\vh_0^\pm\in H^{s}(\Omega_{0}^\pm).
\end{align*}
Furthermore assume that there exists $c_0>0$ so that
\begin{itemize}
\item[1.] $-(1-2c_0)\le f_0\le (1-2c_0)$;

\item[2.] $\Lambda(\vh_0^\pm,[\vu_0])\ge 2c_0$.
\end{itemize}
Then there exists $T>0$ such that the system (\ref{cvs})-(\ref{cvsi}) admits a unique solution $(f, \vu, \vh)$ in $[0,T]$ satisfying
\begin{itemize}
\item[1.] $f\in L^\infty([0,T), H^{s+\f12}(\BT))$;

\item[2.] $\vu^\pm,\,\vh^\pm\in L^\infty\big(0,T;H^{s}(\Omega^\pm_t)\big)$;

\item[3.] $-(1-c_0)\le f\le (1-c_0)$;

\item[4.] $\Lambda(\vh^\pm,[\vu])\ge c_0$.
\end{itemize}
\end{theorem}

Now let us present main ideas of our proof, which are motivated by recent important progress  on the well-posedness for the water-wave problem \cite{Wu1, Wu2, Amb, Lan, Lin, CSH, ZZ}, especially \cite{SZ1, SZ2, SZ3}.

A key idea is to consider the evolution of the unknowns(the free surface, the normal velocity etc.) defined on the free surface and the motion of the fluid in the interior simultaneously. For this end, we will derive  an important evolution equation of the scaled normal velocity defined by
\begin{align*}
\theta(t,x')\eqdefa\vu^\pm(t,x',f(t,x'))\cdot\vN(t,x'),
\end{align*}
which satisfies
\begin{align}\label{eq:surface}
\left\{
\begin{array}{l}
\partial_tf=\theta,\\
\partial_t\theta=\mathcal{A}f+\mathfrak{g},
\end{array}\right.
\end{align}
where $\mathfrak{g}$ denotes the lower order nonlinear terms and
\beno
\mathcal{A}f=-2(w_1\partial_1\theta+w_2\partial_2\theta)+\sum_{i,j=1,2}(-w_iw_j-v_iv_j+\f12h^+_ih^+_j+\f12h^-_ih^-_j)\partial_i\partial_jf,
\eeno
with $(w_1, w_2, w_3)=\f12(\vu^++\vu^-)|_{\Gamma(t)}$ and $(v_1,v_2,v_3)=\f12(\vu^+-\vu^-)=\f12[\vu]$.

The most important finding of this work is that
the system (\ref{eq:surface}) is strictly hyperbolic under the weak stability condition ({\ref{stability-weak}}) in the following sense:  $f$ satisfies a second order
equation in the form
\begin{align*}
D_t^2f=\f12\sum_{i=1,2}(-2v_iv_j+h^+_ih^+_j+h^-_ih^-_j)\partial_i\partial_jf+\cdots,
\end{align*}
where $D_t=\partial_t+w_1\partial_1+w_2\partial_2$, and the principal symbol of the operator is
\begin{align*}
(v_i\xi_i)^2-\f12\big((h^+_i\xi_i)^2+(h^-_i\xi_i)^2\big),
\end{align*}
which is strictly negative under (\ref{stability-weak}).

The motion of the fluid will be described by the vorticity equations. With the vorticity and current, the velocity and magnetic field are recovered by solving the
div-curl system in a finite strip. To ensure the existence and uniqueness of the solution of the div-curl system, we need to introduce the suitable compatiblity conditions on the vorticity, and prescribe a value on the average of the tangential components on the fixed boundary.

To estimate the nonlinear term $\mathfrak{g}$, we  need to study the estimates in Sobolev spaces of the Dirichlet-Neumann(DN) operator. Motivated by \cite{AM}, we will use the paradifferential operator tools to give the precise estimate for the DN operator,  especially on the dependence of regularity of the free surface.

The construction of the approximate solution is completed by introducing the suitable linearization of the system and the iteration map.
We proved that the approximate solution sequence is a Cauchy sequence in the lower order Sobolev spaces. Thus, we can obtain a limit system.
The question of whether the limit is equivalent to the origin system is also highly nontrivial.

\medskip
We believe that our method can be applied to solve the plasma-vacuum interface problem in ideal incompressible MHD and the other related free boundary problems(see \cite{HL} for example). The well-posedness of the linearized plasma-vacuum interface problem has been proved by Morando, Trakhinin and Trebeschi \cite{Mo2}.

\vspace{0.2cm}

The rest of this paper is organized as follows. In section 2, we will introduce the reference domain and harmonic coordinate used in this paper.
In section 3, we introduce the Dirichlet-Neumann operator and present the estimates in Sobolev spaces. In section 4, we solve the div-curl system.
In section 5, we reformulate the system into a new formulation. In section 6, we reformulate the weak stability condition and study linear stability. In section 7, we present the uniform estimates for the linearized system. Section 8-Section 10 are devoted to the existence and uniqueness of the solution.
Section 11 is an appendix, in which we introduce the paradifferential operator and present an elliptic estimate in a strip.

\section{Reference domain and harmonic coordinate}

Motivated by  \cite{SZ1},  we introduce a fixed reference domain in order to
solve the free boundary problem. Let $\Gamma_*$ be a fixed graph given by
\begin{align*}
\Gamma_*=\Big\{(y_1,y_2,y_3):y_3=f_*(y_1,y_2)\Big\}.
\end{align*}
The reference domain $\Om_*^{\pm}$ is given by
\begin{align*}
\Om_*=\mathbb{T}^2\times(-1,1),\quad\Om_*^{\pm}=\Big\{ y \in \Om_*| y_3 \gtrless f_*(y')\Big\}.
\end{align*}

We will seek the free boundary which lies in a neighborhood of the reference domain. For this, we define
\begin{align*}
\Upsilon(\delta,k)&\eqdefa\Big\{f\in H^k(\mathbb{T}^2): \|f-f_*\|_{H^k(\mathbb{T}^2)}\le \delta \Big\}.
\end{align*}
For $f\in \Upsilon(\delta,k)$, we can define the graph $\Gamma_f$ as
\[
\Gamma_f\eqdefa\left\{x\in \Om_t| x_3=f(t,x'), \int_{\mathbb{T}^2}f(t,x')\dx'=0 \right\}.
\]
The graph $\Gamma_f$ separates $\Omega_t$ into two parts:
\[
\Om_f^{+}=\Big\{ x \in \Om_t| x_3 > f(t,x')\Big\}, \quad \Om_f^{-}=\Big\{ x \in \Om_t| x_3 < f(t,x')\Big\}.
\]
We denote
\[\vN_f\triangleq(-\partial_1f, -\partial_2f, 1),\quad \vn_f\triangleq\vN_f/\sqrt{1+|\nabla f|^2}.\]
That is, $\vN_f$ is the outward normal vector of $\Om_f^-$.\medskip

Now we introduce the harmonic coordinate. Given $f\in \Upsilon(\delta,k)$, we define a map $\Phi_f^\pm$ from $\Omega_*^\pm$ to $\Omega_f^\pm$
by harmonic extension:
\beq\left\{\begin{split}
&\Delta_y \Phi_f^\pm=0,\qquad \text{for } y\in \Omega_*^\pm,\\
&\Phi_f^\pm(y',f_*(y'))=(y',f(y')), \quad y'\in\mathbb{T}^2,\\
&\Phi_f^\pm(y',\pm1)=(y',\pm1), \quad y'\in\mathbb{T}^2.\\
\end{split}\right.\eeq

Given $\Gamma_*$, there exists $\delta_0=\delta_0(\|f_*\|_{W^{1,\infty}})>0$ so that $\Phi_f^\pm$ is a bijection when $\delta\le \delta_0$.
Then we can define an inverse map $\Phi_f^{\pm-1}$ from $\Omega_f^\pm$ to $\Omega_*^\pm$ such that
\beno
 \Phi_f^{\pm-1}\circ\Phi_f^\pm=\Phi_f^\pm\circ\Phi_f^{\pm-1}=\mathrm{Id}.
\eeno

Let us state some basic inequalities in different coordinates. The proof is standard, thus we omit it.

\begin{lemma}\label{lem:basic}
Let $f\in \Upsilon(\delta_0,s-\f12)$ for $s\ge 3$. Then there exists a constant $C$ depending only on $\delta_0$ and  $\|f_*\|_{H^{s-\f12}}$ so that

\begin{itemize}
\item[1.] If $u\in H^{\sigma}(\Om_f^\pm)$ for $\sigma\in [0,s]$, then
\beno
&&\|u\circ\Phi_f^\pm\|_{H^\sigma(\Om^\pm_*)}\le C\|u\|_{H^\sigma(\Om_f^\pm)}.
\eeno
\item[2.] If $u\in H^{\sigma}(\Om_*^\pm)$ for $\sigma\in [0,s]$, then
\beno
\|u\circ\Phi_f^{\pm-1}\|_{H^{\sigma}(\Om_f^\pm)}\le C\|u\|_{H^\sigma(\Om_*^\pm)}.
\eeno

\item[3.] If $u, v\in H^{\sigma}(\Om_*^\pm)$ for $\sigma\in [2,s]$, then
\beno
\|uv\|_{H^\sigma(\Omega_f^\pm)}\le C\|u\|_{H^\sigma(\Omega_f^\pm)}\|v\|_{H^\sigma(\Omega_f^\pm)}.
\eeno
\end{itemize}
\end{lemma}

Let us conclude this section by introducing some notations.
\vspace{0.1cm}

We will use $x=(x_1,x_2,x_3)$ or $y=(y_1,y_2,y_3)$ to denote the coordinates in the fluid region,
and use $x'=(x_1,x_2)$ or $y'=(y_1, y_2)$ to denote the natural coordinates on the interface or on the top/bottom boundary.

For a function $g:\Omega\to\mathbb{R}$, we denote $\nabla g=(\partial_1g,\partial_2g,\partial_3g)$, and for a function $\eta:\BT\to\mathbb{R}$,
$\nabla\eta=(\partial_1\eta,\partial_2\eta)$. For a function $g:\Omega_f^\pm\to\mathbb{R}$, we can define its trace on $\Gamma_f$, which is denoted by $\underline{g}(x')$. Thus, for $i=1,2$,
\beno
\partial_{i}\underline{g}(x')=\partial_ig(x',f(x'))+\partial_3g(x',f(x'))\partial_if(x').
\eeno

We denote by $\|\cdot\|_{H^s(\Om)}$ the Sobolev norm in $\Om$, and by $\|\cdot\|_{H^s}$  the Sobolev norm in $\T^2$.
We define
\beno
H_0^s(\T^2)\eqdefa H^s(\T^2)\cap\Big\{\phi\in L^2(\T^2):\int_{\T^2}\phi(x')dx'=0\Big\}.
\eeno

\section{Dirichlet-Neumann(DN) operator}

In the sequel, we always assume that there exists a constant $c_0$ so that
\ben\label{ass:f}
&-(1-c_0)\le f(x')\le (1-c_0)\quad \textrm{for any }x'\in \T^2.
\een

\subsection{Definition of DN operator}

For any $g(x')=g(x_1,x_2)\in H^k(\mathbb{T}^2)$, we can view $g$ as a function on $\Gamma_f$ and
then denote by $\mathcal{H}_f^\pm g$ the harmonic extension to $\Omega^\pm_f$, i.e.,
\beq\left\{\begin{split}
&\Delta \mathcal{H}_f^\pm g =0,\qquad \text{for } x\in \Omega_f^\pm,\\
&(\mathcal{H}_f^\pm g)(x',f(x'))=g(x'), \quad x'\in\mathbb{T}^2,\\
&\partial_3\mathcal{H}_f^\pm g(x',\pm1)=0, \quad x'\in\mathbb{T}^2.
\end{split}\right.\eeq
The Dirichlet-Neumann operator is defined by
\begin{align}
  \mathcal{N}^\pm_fg\eqdefa\mp\vN_f\cdot(\nabla\mathcal{H}^\pm_fg)\big|_{\Gamma_f}.
\end{align}
We also define
\begin{align}
\widetilde{\mathcal{N}}_fg\eqdefa\mathcal{N}^+_fg+\mathcal{N}^-_fg.
\end{align}

The Dirichlet-Neumann operator has the following basic properties(see \cite{Lan} for example).

\begin{lemma}
\label{lem:DN}
It holds that
\begin{itemize}

\item[1.] $\mathcal{N}^\pm_f$ is a self-adjoint operator:
\[
(\mathcal{N}^\pm_f\psi,\phi)=(\psi,\mathcal{N}^\pm_f\phi),\quad\forall \phi,\, \psi\in H^\f12(\T^2);
\]

\item[2.] $\mathcal{N}^\pm_f$ is a positive operator:
\[
(\mathcal{N}^\pm_f\phi,\phi)=\|\na\mathcal{H}_f^\pm\phi\|_{L^2(\Omega_f)}^2\ge 0,\quad \forall \phi\in H^\f12(\T^2);
\]
Especially, if $\int_{\T^2}\phi(x')dx'=0$, there exists $c>0$ depending on $c_0, \|f\|_{W^{1,\infty}}$ such that
\[
(\mathcal{N}^\pm_f\phi,\phi)\ge c\|\mathcal{H}_f^\pm\phi\|_{H^1(\Omega_f)}^2\ge c\|\phi\|_{H^\f12}^2.
\]

\item[3.] $\mathcal{N}^\pm_f$ is a bijection from $H^{k+1}_0(\T^2)$ to $H^{k}_0(\T^2)$ for $k\ge 0$.
\end{itemize}
\end{lemma}

\subsection{Paralinearization of DN operator}

Motivated by \cite{AM}, we use the paradifferential operator to study the DN operator.
In this subsection, we will frequently use notations introduced in the appendix.

In terms of $\Psi$, the DN operator $\mathcal{N}^-_f\psi$ can be written as
\beno
\mathcal{N}^-_f\psi=\Big(\f{1+|\nabla \rho_\delta|^2}{\pa_z \rho_\delta}\pa_z\Psi-\nabla\rho_\delta \cdot \nabla\Psi\Big)\Big|_{z=0}.
\eeno
We denote
\beno
\zeta_1(x)\triangleq\f{1+|\nabla \rho_\delta|^2}{\pa_z \rho_\delta}\big|_{z=0}=\f {1+|\na f|^2} {\pa_z\rho_\delta|_{z=0}},\quad \zeta_2(x)\triangleq\na\rho_\delta\big|_{z=0}=\na f(x).
\eeno
It is easy to show that for $s>\f32$,
\ben\label{eq:DN-zeta-Hs}
\|\zeta_1-1\|_{H^{s-\f12}}+\|\zeta_2\|_{H^{s-\f12}}\le C\big(c_0, \|f\|_{H^{s+\f12}}\big).
\een

Using Bony's decomposition (\ref{Bony}), we can decompose $\mathcal{N}^-_f$ as
\begin{align*}
\mathcal{N}^-_f\psi=&\pa_z\Psi+T_{\zeta_1-1}\pa_z\Psi+T_{\pa_z\Psi}(\zeta_1-1)+R(\zeta_1-1,\pa_z\Psi)-T_{i\zeta_2\cdot \xi}\Psi\\
&-T_{\nabla\Psi}\cdot\zeta_2-R(\zeta_2,\nabla\Psi)\big|_{z=0}.
\end{align*}
Replacing $\pa_z\Psi$ by $T_A\Psi$, we get
\ben\label{eq:DN-para-m}
\mathcal{N}^-_f\psi=T_{\lambda}\psi+R_f^-\psi,
\een
where the symbol $\lambda(x,\xi)$ of the leading term is given by
\beno
\lambda(x,\xi)=\zeta_1A-i\zeta_2\cdot \xi\big|_{z=0}=\sqrt{(1+|\na f|^2)|\xi|^2-(\na f\cdot\xi)^2}.
\eeno
Obviously, $\lambda\in \Gamma_{\e}^{1}(\T^2)$ with the bound
\ben\label{eq:lambda-est}
\|\lambda\|_{M_\e^1}\le C\big(c_0, \|f\|_{H^{s+\f12}}\big)\quad \text{for any}\quad \e\in \big(0,s-\f32\big).
\een

The remainder $R_f^-$ of the DN operator is given by
\begin{align}
R_f^-\psi
=&\Big[\big(T_{\zeta_1}T_A-T_{\zeta_1 A}\big)\Psi-T_{\zeta_1}(\pa_z-T_A)\Psi\nonumber\\
 &+\big(S_2(\pa_z\Psi)+T_{\pa_z\Psi}(\zeta_1-1)+R(\zeta_1-1,\pa_z\Psi)-T_{\nabla \Psi}\cdot\zeta_2-R(\nabla\Psi,\zeta_2)\big)\Big]\bigg|_{z=0}\nonumber\\
\triangleq & R_{1,f}^-\Psi+R_{2,f}^-\Psi+R_{3,f}^-\Psi.\label{eq:DN-R}
\end{align}

Similarly, we have
\ben\label{eq:DN-para-p}
\mathcal{N}^+_f\psi=T_{\lambda}\psi+R_f^+\psi,
\een
where $R_f^+$ has a similar representation as $R_f^-$.

\subsection{Sobolev estimates of DN operator}

Let us first prove the following Sobolev estimate for the remainder.
In the sequel, we denote by $K_{s,f}$  a constant depending on $c_0$
and $\|f\|_{H^{s}}$, which may be different from line to line.

\begin{lemma}\label{lem:remainder}
If $f\in H^{s+\f12}(\T^2)$ for $s>\f52$, then it holds that for any $\sigma\in \big[\f12,s-\f12\big]$,
\beno
\|R_f^\pm\psi\|_{H^\sigma}\le K_{s+\f12,f}\|\psi\|_{H^\sigma}.
\eeno
\end{lemma}

\begin{proof}
Due to $s>\f52$, we know from (\ref{eq:lambda-est}) that
$A\in \Gamma_1^1(\T^2)$. Then it follows from Proposition \ref{prop:symcal} and Proposition \ref{prop:elliptic} that
\beno
\|R_{1,f}^-\psi\|_{H^\sigma}\le K_{s+\f12,f}\|\psi\|_{H^\sigma}.
\eeno
Thanks to $s>\f52$ and $\sigma\in \big[\f12,s-\f12\big]$, we infer from Lemma \ref{lem:bony} and Proposition \ref{prop:elliptic} that
\begin{align*}
\|R_{3,f}^-\psi\|_{H^\sigma}\le& K_{s+\f12,f}\|\na_{x,z}\Psi(x,0)\|_{H^{\sigma-1}}\le K_{s+\f12,f}\|\psi\|_{H^\sigma}.
\end{align*}
Due to $s>\f52$ and $\sigma\ge \f12$, we may apply (\ref{eq:W-est}) with $\e=1$ to obtain
\begin{align*}
\|R_{2,f}^-\psi\|_{H^\sigma}\le K_{s-\f12,f}\|(\pa_z-T_A)\Psi(x,0)\|_{H^{\sigma}}\le K_{s+\f12,f}\|\psi\|_{H^\sigma}.
\end{align*}
The proof is the same for $R_f^+$.
\end{proof}

\begin{proposition}\label{prop:DN-Hs}
If $f\in H^{s+\f12}(\T^2)$ for $s>\f52$, then it holds that for any $\sigma\in \big[-\f12,s-\f12\big]$,
\beno
&&\|\mathcal{N}^\pm_f\psi\|_{H^{\sigma}}\le  K_{s+\f12,f}\|\psi\|_{H^{\sigma+1}}.
\eeno
Moreover, it holds that for any $\sigma\in \big[\f12,s-\f12\big]$,
\beno
\|\big(\mathcal{N}^+_f-\mathcal{N}^-_f\big)\psi\|_{H^\sigma}\le K_{s+\f12,f}\|\psi\|_{H^\sigma}.
\eeno
\end{proposition}

\begin{proof}
In the case of $\sigma\in \big[\f12,s-\f12\big]$, the first inequality of the lemma follows from Proposition \ref{prop:symcal} and Lemma \ref{lem:remainder}.
The case of $\sigma=-\f12$ follows from Lemma \ref{lem:DN}. The other cases can be deduced by the interpolation.

Using the formula
\beno
\big(\mathcal{N}^+_f-\mathcal{N}^-_f\big)\psi=\big(R^+_f-R^-_f\big)\psi,
\eeno
the second inequality  follows easily from Lemma \ref{lem:remainder}.
\end{proof}

Next we study the inverse of $\mathcal{N}_f^\pm$.

\begin{proposition}\label{prop:DN-inverse}
If $f\in H^{s+\f12}(\T^2)$ for $s>\f52$, then it holds that for any $\sigma\in \big[-\f12,s-\f12\big]$,
\beno
&&\|\mathcal{G}^\pm_f\psi\|_{H^{\sigma+1}}\le  K_{s+\f12,f}\|\psi\|_{H^{\sigma}},
\eeno
where $\mathcal{G}^\pm_f\triangleq\big(\mathcal{N}^\pm_f\big)^{-1}$.
\end{proposition}

\begin{proof}
Let $\psi=\mathcal{N}^+_f\phi\in H^{\sigma}_0(\T^2)$, i.e., $\phi=\mathcal{G}^+_f\psi$.
Then we have
\beno
\psi=T_\lambda\phi+R_f^+\phi,
\eeno
which gives
\beno
\phi=T_{\lambda^{-1}}\psi-\big(T_{\lambda^{-1}}T_{\lambda}-1)\phi-T_{\lambda^{-1}}R^+_f\phi.
\eeno
Then it follows from Proposition \ref{prop:symcal} and Lemma \ref{lem:remainder} that for $\sigma\ge \f12$,
\begin{align*}
\|\phi\|_{H^{\sigma+1}}\le  K_{s-\f12,f}\|\psi\|_{H^\sigma}+K_{s+\f12,f}\|\phi\|_{H^\sigma}.
\end{align*}
On the other hand, Lemma \ref{lem:DN}  implies that
\beno
\|\phi\|_{H^\f12}\le C\big(c_0, \|f\|_{W^{1,\infty}}\big)\|\psi\|_{H^{-\f12}}.
\eeno
Thus, we get by the interpolation that
 \begin{align*}
\|\phi\|_{H^{\sigma+1}}\le&  K_{s-\f12,f}\|\psi\|_{H^\sigma}+K_{s+\f12,f}\|\phi\|_{H^\f12}+\f12\|\phi\|_{H^{\sigma+1}}\\
\le& K_{s+\f12,f}\|\psi\|_{H^\sigma}+\f12\|\phi\|_{H^{\sigma+1}},
\end{align*}
which implies the desired inequality for $\sigma\in [\f12,s-\f12]$. The case of $\sigma\in [-\f12,\f12)$
can be proved by the interpolation.
\end{proof}

\subsection{Commutator estimates of DN operator}
We present some commutator estimates of DN operator.
Although they will not be used in this paper, they are independent of interest and may be useful for the zero-surface tension limit problem.

\begin{proposition}\label{prop:DN-com}
If $f\in H^{s+\f12}(\T^2)$ for $s>\f52$, then it holds that for any $\sigma\in \big[\f32,s-\f12\big]$,
\beno
&&\big\|[\pa_i,\mathcal{N}^\pm_f]\psi\big\|_{H^{\sigma-1}}\le  K_{s+\f12,f}\|\psi\|_{H^{\sigma}}.
\eeno
For any $\sigma\in \big(1,s-\f12\big]$, we have
\beno
\big\|[a,\mathcal{N}^\pm_f]\psi\big\|_{H^{\sigma}}\le   K_{s+\f12,f}\|a\|_{H^{\sigma+1}}\|\psi\|_{H^{\sigma}}.
\eeno
\end{proposition}

\begin{proof}
Let us first prove the first commutator estimate. We get by (\ref{eq:DN-para-m}) and (\ref{eq:DN-para-p}) that
\beno
[\pa_i,\mathcal{N}^\pm_f]\psi=[\pa_i,T_\lambda]\psi+\pa_iR_f^\pm\psi-R_f^\pm\pa_i\psi,
\eeno
which together with Proposition \ref{prop:symcal} and Lemma \ref{lem:remainder} gives the first inequality of the lemma.

By (\ref{eq:DN-para-m}) and (\ref{eq:DN-para-p}) again, we have
\beno
[a,\mathcal{N}^\pm_f]\psi=[a, T_\lambda]\psi+aR_f^\pm\psi-R^\pm_f(a\psi).
\eeno
For $\sigma>1$, $H^\sigma(\T^2)$ is an algebra. Thus, we infer from Lemma \ref{lem:remainder} that
\ben\label{eq:a-com}
\|aR_f^\pm\psi\|_{H^\sigma}+\|R^\pm_f(a\psi)\|_{H^\sigma}\le K_{s+\f12,f}\|a\|_{H^\sigma}\|\psi\|_{H^{\sigma}}.
\een
We write
\begin{align*}
[a, T_\lambda]\psi=[T_a, T_\lambda]\psi+\big(a-T_a)T_\lambda\psi-T_\lambda(a-T_a)\psi.
\end{align*}
By Proposition \ref{prop:symcal} and Sobolev embedding, we get
\begin{align*}
\|[T_a, T_\lambda]\psi\|_{H^\sigma}\le  K_{s+\f12,f}\|a\|_{W^{1,\infty}}\|\psi\|_{H^{\sigma}}\le K_{s+\f12,f}\|a\|_{H^{\sigma+1}}\|\psi\|_{H^{\sigma}}.
\end{align*}
Using Bony's decomposition (\ref{Bony}) and Lemma \ref{lem:bony}, we have
\begin{align*}
\|(a-T_a)T_\lambda\psi\|_{H^\sigma}\le C\|a\|_{H^{\sigma+1}}\|T_\lambda\psi\|_{H^{\sigma-1}}\le K_{s-\f12,f}\|a\|_{H^{\sigma+1}}\|\psi\|_{H^{\sigma}}.
\end{align*}
Similarly, we have
\begin{align*}
\|T_\lambda(a-T_a)\psi\|_{H^\sigma}\le K_{s-\f12,f}\|(a-T_a)\psi\|_{H^{\sigma+1}}\le  K_{s-\f12,f}\|a\|_{H^{\sigma+1}}\|\psi\|_{H^{\sigma}}.
\end{align*}
This shows that
\begin{align*}
\big\|[a, T_\lambda]\psi\big\|_{H^\sigma}\le K_{s+\f12,f}\|a\|_{H^{\sigma+1}}\|\psi\|_{H^{\sigma}},
\end{align*}
which along with (\ref{eq:a-com}) gives the second commutator estimate.
\end{proof}

Finally, we study the commutator estimate between the DN operator and the time derivative.

\begin{proposition}\label{prop:DN-com-time}
If $f\in H^{s+\f12}(\T^2)$ and $\pa_tf\in H^{s-\f12}(\T^2)$ for $s>\f52$, then it holds that for any $\sigma\in \big[\f32,s-\f12\big]$,
\beno
&&\big\|[\pa_t,\mathcal{N}^\pm_f]\psi\big\|_{H^{\sigma-1}}\le  C\big(c_0, \|f\|_{H^{s+\f12}}, \|\pa_tf\|_{H^{s-\f12}}\big)\|\psi\|_{H^{\sigma}}.
\eeno
\end{proposition}

\begin{proof}
We first get by (\ref{eq:DN-para-p}) that
\beno
[\pa_t,\mathcal{N}^-_f]\psi=T_{\pa_t\lambda}\psi+\pa_tR_f^-\psi-
R_f^-\pa_t\psi.
\eeno
It follows from Proposition \ref{prop:symcal} that
\beno
\|T_{\pa_t\lambda}\psi\|_{H^{\sigma-1}}\le C\big(c_0, \|f\|_{H^{s-\f12}}, \|\pa_tf\|_{H^{s-\f12}}\big)\|\psi\|_{H^{\sigma}}.
\eeno

We denote
\beno
\Psi^t(t,x,z)=\mathcal{H}_f^-\big(\pa_t\psi)(t,x,\rho_\delta\big),\quad \Psi_t(t,x,z)=(\pa_t\mathcal{H}_f^-\big(\psi))(t,x,\rho_\delta\big).
\eeno
Thanks to (\ref{eq:DN-R}), we find that
\begin{align*}
[\pa_t,R_f^-]\psi=&(T_{\pa_t\zeta_1}T_A-T_{\pa_t\zeta_1A})\Psi\big|_{z=0}
+(T_{\zeta_1}T_{\pa_tA}-T_{\zeta_1\pa_tA})\Psi|_{z=0}\\
&+(T_{\zeta_1}T_A-T_{\zeta_1A})(\pa_t\Psi-\Psi^t)\big|_{z=0},
\end{align*}
where we have
\beno
\pa_t\Psi-\Psi^t=\Psi_t-\Psi^t+\f {\pa_t\rho_\delta} {\pa_z\rho_\delta}
\pa_z\Psi.
\eeno
Then we infer from Proposition \ref{prop:symcal} and Lemma \ref{lem:ellip-td} that
\beno
\big\|[\pa_t,R_{1,f}^-]\psi\big\|_{H^{\sigma-1}}\le C\big(c_0,\|f\|_{H^{s+\f12}},\|\pa_tf\|_{H^{s-\f12}}\big)\|\psi\|_{H^{\sigma}}.
\eeno

By (\ref{eq:DN-R}) again, we have
\begin{align*}
[\pa_t,R_{2,f}^-]\psi=&T_{\pa_t\zeta_1}(\pa_z-T_A)\Psi|_{z=0}-T_{\zeta_1}T_{\pa_t A}\Psi|_{z=0}\\
&+T_{\zeta_1}(\pa_z-T_A)\big(\pa_t\Psi-\Psi^t\big)|_{z=0}.
\end{align*}
Then by Lemma \ref{lem:ellip-td} and the proof of (\ref{eq:W-est}), we get
\beno
\big\|[\pa_t,R_{2,f}^-]\psi\big\|_{H^{\sigma-1}}\le C\big(c_0,\|f\|_{H^{s+\f12}},\|\pa_tf\|_{H^{s-\f12}}\big)\|\psi\|_{H^{\sigma}}.
\eeno

The estimate of $[\pa_t,R_{3,f}^-]\psi$ can be deduced from Lemma \ref{lem:ellip-td} and Lemma \ref{lem:bony}.
\end{proof}

\begin{lemma}\label{lem:ellip-td}
With the same assumptions as in Proposition \ref{prop:DN-com-time},
let $\delta\Phi(t,x,y)=\mathcal{H}_f^-\big(\pa_t\psi)-\pa_t\mathcal{H}_f^-\big(\psi)$ and $\delta\Psi=\delta\Phi(t,x,\rho_\delta)$.
Then it holds that
\beno
\|\nabla_{x,z}\delta\Psi\|_{X^{\sigma-2}(I)}
\le C\big(c_0,\|f\|_{H^{s+\f12}},\|\pa_tf\|_{H^{s-\f12}}\big)\|\psi\|_{H^{\sigma}}.
\eeno

\end{lemma}

\begin{proof}
It is easy to see that $\delta\Phi$ satisfies
\beno
\left\{\begin{array}{l}
\Delta_{x,y}\delta\Phi=0\quad \text{in}\quad\Omega_f^-,\\
\delta\Phi(t,x,f(t,x))=-\pa_tf\pa_y\Phi(t,x,f(t,x)) \quad \text{for}\quad x\in\mathbb{T}^2,\\
\partial_y\delta\Phi(t,x,-1)=0\quad \text{for}\quad x\in\mathbb{T}^2.
\end{array}\right.
\eeno
It follows from Proposition \ref{prop:elliptic} that
\begin{align*}
\|\nabla_{x,z}\delta\Psi\|_{X^{\sigma-2}(I)}\leq& K_{s+\f12,f}\big\|\f {\pa_tf} {\pa_z\rho_\delta|_{z=0}}(\pa_z\Psi)(t,x,f(t,x))\big\|_{H^{\sigma-1}}\\
\le& C\big(c_0,\|f\|_{H^{s+\f12}},\|\pa_tf\|_{H^{s-\f12}}\big)\|\psi\|_{H^{\sigma}}.
\end{align*}
The proof is finished.\
\end{proof}

\section{Div-curl system}

In this section, we solve the following div-curl system
\beq\label{eq:div-curl}
\left\{\begin{split}
&\curl \vu =\vom,\quad\div \vu=g\quad \text{ in }\quad\Om_f^+,\\
&\vu\cdot\vN_f =\vartheta\quad \text{ on}\quad \Gamma_f,   \\
&\vu\cdot\ve_3 = 0,\quad \int_{\BT} u_i dx'=\alpha_i (i=1,2)\quad \text{ on}\quad\Gamma^{+}.
\end{split}\right.
\eeq

In this section, we assume that $f\in H^{s+\f12}(\T^2)$ for $s\ge 2$ and satisfies (\ref{ass:f}).
Our main result  is stated as follows.

\begin{proposition}\label{prop:div-curl}
Let $\sigma \in [2,s]$ be an integer. Given $\vom, g\in H^{\sigma-1}(\Omega_f^+)$, $\vartheta\in H^{\sigma-\frac12}(\Gamma_f)$ with the compatiblity condition:
\begin{align*}
  \int_{\Om_f^+} g dx=\int_{\Gamma_f} \vartheta ds,
\end{align*}
and $\vom$ satisfies
\begin{align*}
&\div\vom=0\quad \text{in}\quad \Omega_f^+,\quad\int_{\Gamma^+}\om_3dx'=0,
\end{align*}
Then there exists a unique $\vu\in H^{\sigma}(\Omp)$ of the div-curl system (\ref{eq:div-curl}) so that
\begin{align*}
\|\vu\|_{H^{\sigma}(\Omega_f^+)}\le C\big(c_0,\|f\|_{H^{s+\f12}}\big)\left(\|\vom\|_{H^{\sigma-1}(\Omega_f^+)}+\|g\|_{H^{\sigma-1}(\Omega_f^+)}
+\|\vartheta\|_{H^{\sigma-\frac12}(\Gamma_f)}+|\alpha_1|+|\alpha_2|\right).
\end{align*}
\end{proposition}

\begin{remark}
For the compatiblility conditions on $\om$, we have the following geometric interpretation:

Let $\widetilde\omega=\om_1dx_2\wedge dx_3+\om_2dx_3\wedge dx_1+\om_3dx_1\wedge dx_2$. Then $\widetilde\omega$
is a closed 2-form on $\Omega_f^+$. Since $H^2_{dR}(\Omega_f^+)\simeq H^2_{dR}(S^1\times S^1\times [0,1])=\mathbb{R}$, all closed 2-form which
is not exact must be $cdx_1\wedge dx_2+ d\sigma$ for some $c\in\mathbb{R}$ and 1-form $\sigma$.
Thus, if $\int_{\Gamma^+}\om_3dx_1dx_2=0,$ then $\widetilde\omega$ is exact and $\vom$ must be a curl of
some vector field $\vu$.

\end{remark}

The proof of the proposition is based on the following lemmas.

\begin{lemma}\label{lem:div-curl-1}
Let $\sigma \in [1,s]$ be an integer. Given $g\in H^{\sigma-1}(\Omega_f^+), \vartheta\in H^{\sigma-\f12}(\Gamma_f),\nu\in H^{\sigma-\f12}(\Gamma^+)$  with the following compatibility condition
\begin{align*}
\int_{\Omega_f^+}gdx=\int_{\Gamma_f}\vartheta ds+\int_\BT\nu dx',
\end{align*}
there exists a unique periodic(in $x'$) solution $\phi\in H^{\sigma+1}(\Omega_f^+)$ up to a constant of the elliptic equation
\beno\left\{\begin{split}
&\Delta\phi =g\quad \text{ in}\quad\Om_f^+,\\
&\vN_f\cdot\nabla\phi=\vartheta\quad \text{on}\quad \Gamma_f,\\
&\partial_3\phi=\nu\quad \text{ on}\quad\Gamma^{+}.
\end{split}\right.\eeno
Moreover, we have
\begin{align*}
\|\nabla\phi\|_{H^{\sigma}(\Omega_f^+)}\le C\big(c_0,\|f\|_{H^{s+\f12}}\big)\left(\|g\|_{H^{\sigma-1}(\Omega_f^+)}
+\|\vartheta\|_{H^{\sigma-\frac12}(\Gamma_f)}+\|\nu\|_{H^{\sigma-\frac12}(\Gamma^+)}\right).
\end{align*}
\end{lemma}

\begin{proof}
The proof is standard by using Lax-Milgram theorem and regularity theory for elliptic equation(see \cite{CS1, Lan}, for example).
\end{proof}

\begin{lemma}\label{lem:ellip-unique}
Given $\tilde\theta\in H^\f12(\Gamma_f)$, there exists a unique solution  $\vv\in H^1(\Om_f^+)$ to the system
\beno\left\{\begin{split}
&\curl \vv =0,\quad \div \vv=0\quad\text{ in}\quad\Om_f^+,\\
&\vv\cdot\vN_f = \tilde\theta\quad \text{ on}\quad\Gamma_f,   \\
&\vv\cdot\ve_3 = 0,\quad\int_{\BT} v_i dx'=0 (i=1,2)\quad \text{ on}\quad\Gamma^{+}.
\end{split}\right.
\eeno
\end{lemma}

\begin{proof}
Without loss of generality, we assume $\tilde\theta=0$.
From $\curl\vv=0$, we know that there exists a scalar function $\phi$ such that $\vv=\nabla\phi$.
Let $\zeta(x)=\phi(x_1+2\pi,x_2,x_3)-\phi(x_1,x_2,x_3)$. Then $\nabla \zeta=0$, thus $\zeta(x)$ is a constant.
On the other hand, we have
\begin{align*}
0&=\int_{\BT\times\{x_3=1\}}v_1dx'=\int_{\BT\times\{x_3=1\}} \partial_1\phi dx'\\
&=\int_0^1\Big(\phi(2\pi,x_2,1)-\phi(0,x_2,1)\Big)dx'\\
&=\zeta.
\end{align*}
This means that $\phi$ is periodic in $x_1$. Similarly, $\phi$ is also periodic in $x_2$.  Thus, $\phi\in H^2(\Omega_f^+)$ and is harmonic in $\Omega_f^+$ with homogeneous Neumann boundary condition. This implies the uniqueness of the solution from Lemma \ref{lem:div-curl-1}.
\end{proof}

\begin{lemma}\label{lem:div-curl-2}
Let $\sigma \in [2,s]$ be an integer. If $\vom\in H^{\sigma-1}(\Om_f^+)$ satisfies
\begin{align*}
&\div\vom=0\quad \text{in}\quad \Omega_f^+,\quad \int_{\Gamma^+}\om_3dx'=0,
\end{align*}
then there exists $\vu\in H^{\sigma}(\Om_f^+)$ such that
\beq\label{eq:div-curl-H}
\left\{\begin{split}
&\curl\vu=\om,\quad\div\vu=0\quad \text{in}\quad\Om_f^+,\\
&\vu\cdot\vN_f=0\quad \text{on}\quad\Gamma_f,   \\
&\vu\cdot\ve_3=0\quad\text{on}\quad\Gamma^{+}.
\end{split}\right.
\eeq
Moreover, we have
\beno
\|\vu\|_{H^{\sigma}(\Om_f^+)}\le C\big(c_0,\|f\|_{H^{s+\f12}}\big)\|\vom\|_{H^{\sigma-1}(\Om_f^+)}.
\eeno

\begin{proof}
In the case when $\Om_f^+$ is flat, the system of (\ref{eq:div-curl-H}) can be explicitly solved by transforming the system into an ODE system.

In general case, we follow the extension argument from \cite{CS1}.
Let us give a sketch, see section 5.2.2 in \cite{CS1} for the details. Let $\overline{\om}$ be a divergence-free extension of $\om$ to $\Om_f=\Om_f^+\cup \Om_f^-$,
which is defined as follows
\beno
\overline{\vom}=\na \phi\quad \text{in}\quad \Om_f^-,
\eeno
where $\phi$ solves
\beno\left\{\begin{split}
&\Delta\phi=0\quad \text{ in}\quad\Om_f^-,\\
&\vN_f\cdot\nabla\phi=\vN_f\cdot\vom\quad \text{on}\quad \Gamma_f,\\
&\partial_3\phi=0\quad \text{ on}\quad\Gamma^{-}.
\end{split}\right.\eeno
Then we introduce $\vv$ by solving the following system in $\Om=\T^2\times [-1,1]$
\beno
\left\{\begin{split}
&\curl\vv=\overline{\vom},\quad\div\vv=0\quad \text{in}\quad\Om,\\
&\vv\cdot\ve_3=0\quad\text{on}\quad\Gamma^{\pm}.
\end{split}\right.
\eeno
Since $\vv$ does not satisfy our desired regularity and the boundary condition, we need to subtract the nonregular part from $\vv$.
For this, we introduce $p$ which solves
\beno
\left\{\begin{split}
&\Delta p=0\quad \text{ in}\quad\Om_f^+,\\
&\vN_f\cdot\nabla p=-\vN_f\cdot\vv\quad \text{on}\quad \Gamma_f,\\
&\partial_3p=0\quad\text{ on}\quad\Gamma^{+}.
\end{split}\right.
\eeno
Then $\vu=\vv+\na p$ is our desired solution.
 \end{proof}

\end{lemma}

Now let us prove Proposition \ref{prop:div-curl}.

\begin{proof}
The uniqueness of the solution is a direct consequence of Lemma \ref{lem:ellip-unique}.
Let us prove the existence of the solution. Let $\vu_1$ be a solution of the system (\ref{eq:div-curl-H}) determined by Lemma \ref{lem:div-curl-2}.
By Lemma \ref{lem:div-curl-1}, we can find $\phi$ so that
\begin{align*}
\left\{
\begin{array}{l}
\Delta\phi=g\quad\text{ in}\quad\Om_f^+,\\
\vN_f\cdot\nabla\phi=\vartheta-\beta_1\partial_1f-\beta_2\partial_2f\quad\text{on}\quad \Gamma_f,   \\
\partial_3\phi=0\quad \text{on}\quad\Gamma^{+}.
\end{array}\right.
\end{align*}
Let $\vu=\vu_1+\nabla\phi+(\beta_1,\beta_2,0)$. Then $\vu$ is a unique solution of the system (\ref{eq:div-curl})  and satisfies the desired bound.
\end{proof}

\section{Reformulation of the problem}\label{Sec:Reform}

In this section, we will reformulate the system (\ref{cvs})-(\ref{cvsi}) into a new formulation, which consists of the evolution equations of the follow quantities:
\begin{itemize}
\item The height function of the interface: $f$;
\item The scaled normal velocity on the interface: $\theta=\vu^\pm\cdot\vN_f$;
\item The curl part of velocity and magnetic field in the fluid region: $\vom=\nabla\times\vu$, $\vj=\nabla\times\vh$;
\item The average of tangential part of velocity and magnetic field on top and bottom fixed boundary:
\begin{align*}
\beta_i^\pm(t)=\int_{\BT}u_i^\pm(t,x',\pm1)dx',\quad \gamma_i^\pm(t)=\int_{\BT}h_i^\pm(t,x',\pm1)dx'(i=1,2).
\end{align*}
\end{itemize}

\subsection{Evolution of the scaled normal velocity}
We define
\begin{align}
\theta(t,x')\eqdefa\vu^\pm(t,x',f(t,x'))\cdot\vN_f(t,x').
\end{align}
Thus, we have
\begin{align}\label{eq:form:f}
\partial_tf(t,x')=\theta(t,x').
\end{align}
Clearly, $(1+|\nabla f|^2)^{-1/2}\theta$ is the normal component of the fluid velocity on the interface $\Gamma_f$.
In this subsection, we will derive the evolution equation of $\theta$.

\begin{lemma}\label{rel-uh}
For $\vu=\vupm,\vhpm$, we have
\begin{align}
&({\vu}\cdot{\nabla\vu})\cdot\vN_f-{\partial_3u}_jN_j({\vu}\cdot\vN_f)\big|_{x_3=f(t,x')}\nonumber\\
&=\underline{u}_1\partial_1(\underline{u}_jN_j)+\underline{u}_2\partial_2(\underline{u}_jN_j)
+\sum_{i,j=1,2}\underline{u}_i\underline{u}_j\partial_i\partial_jf.\end{align}
\end{lemma}

\begin{proof}
A direct calculation shows that
\begin{align*}
&\underline{u}_1\partial_1(\underline{u}_jN_j)+\underline{u}_2\partial_2(\underline{u}_jN_j)-\underline{u}_1\underline{u}_j\partial_1N_j-\underline{u}_2\underline{u}_j\partial_2N_j\\
&={u}_1({\partial_1u}_j+{\partial_3u}_j\partial_1f)N_j+{u}_2({\partial_2u}_j+{\partial_3u}_j\partial_2f)N_j\big|_{x_3=f(t,x')}\\
&=u_1\partial_1u_jN_j+u_2\partial_2u_jN_j+u_3\partial_3u_jN_j+(u_1\partial_1f+u_2\partial_2f-u_3)\partial_3u_jN_j\big|_{x_3=f(t,x')}\\
&=(\vu\cdot\nabla\vu)\cdot\vN_f-\partial_3u_jN_j(\vu\cdot\vN_f)\big|_{x_3=f(t,x')}.
\end{align*}
This implies the lemma by recalling $\vN_f=(-\partial_{1}f,-\partial_{2}f, 1)$.
\end{proof}

Now, let us derive the evolution equation of $\theta$.
Using the first equation of (\ref{cvs}), we deduce from Lemma \ref{rel-uh}(recall $\vhp\cdot\vN_f=0$ on $\Gamma_f$) that
\begin{align}\nonumber
\partial_t\theta=&(\partial_t\vu^++\partial_3\vu^+\partial_tf)\cdot\vN_f+\vu^+\cdot\partial_t\vN_f\big|_{x_3=f(t,x')}\\
=&(-\vup\cdot\nabla\vup+\vhp\cdot\nabla\vhp-\nabla p^++\partial_3\vu^+\partial_tf)\cdot\vN_f-
\vu^+\cdot(\partial_1\partial_tf,\partial_1\partial_tf,0)\big|_{x_3=f(t,x')}\nonumber\\\nonumber
=&\big(-(\vup\cdot\nabla)\vup+\partial_3\vu^+(\vup\cdot\vN_f)\big)\cdot\vN_f+(\vhp\cdot\nabla)\vhp\cdot\vN_f\\
&-\vN_f\cdot\nabla p^+-\vu^+\cdot(\partial_1\theta, \partial_2\theta,0)\big|_{x_3=f(t,x')}\nonumber\\
=&-2(\underline{u}^+_1\partial_1\theta+\underline{u}^+_2\partial_2\theta)-\vN_f\cdot\underline{\nabla p^+}-\sum_{i,j=1,2}
\underline{u}_i^+\underline{u}^+_j\partial_i\partial_jf\nonumber\\
&\quad+\sum_{i,j=1,2}\underline{h}^+_i\underline{h}^+_j\partial_i\partial_jf.\label{eq:theta-tup}
\end{align}
A similar derivation gives
\begin{align}\nonumber
\partial_t\theta=&-2(\underline{u}^-_1\partial_1\theta+\underline{u}^-_2\partial_2\theta)-\vN_f\cdot\underline{\nabla p^-}
-\sum_{i,j=1,2}\underline{u}^-_i\underline{u}^-_j\partial_i\partial_jf\\
&\quad+\sum_{i,j=1,2}\underline{h}^-_i\underline{h}^-_j\partial_i\partial_jf.\label{eq:theta-tdown}
\end{align}
It follows from (\ref{eq:theta-tup}) and (\ref{eq:theta-tdown}) that
\begin{align}\nonumber
&2(\underline{u}^+_1\partial_1\theta+\underline{u}^+_2\partial_2\theta)+\vN_f\cdot\underline{\nabla p^+}
+\sum_{i,j=1,2}\underline{u}_i^+\underline{u}^+_j\partial_i\partial_jf-\sum_{i,j=1,2}\underline{h}^+_i\underline{h}^+_j\partial_i\partial_jf\\
&=2(\underline{u}^-_1\partial_1\theta+\underline{u}^-_2\partial_2\theta)+\vN_f\cdot\underline{\nabla p^-}
+\sum_{i,j=1,2}\underline{u}^-_i\underline{u}^-_j\partial_i\partial_jf-\sum_{i,j=1,2}\underline{h}^-_i\underline{h}^-_j\partial_i\partial_jf.\label{eq:theta-d}
\end{align}

Taking the divergence to the first equation of (\ref{cvs}), we get
\begin{align}
\Delta p^\pm=\mathrm{tr}(\nabla\vh^\pm)^2-\mathrm{tr}(\nabla\vu^\pm)^2.
\end{align}

Recall that $\underline{p}^\pm=p^\pm|_{\Gamma_f}$ and $\mathcal{H}^\pm_f$
be the harmonic extension from $\Gamma_f$ to $\Omega^\pm_f$.
Let $p_{\vupm_1, \vupm_2}$ be the solution of elliptic equation
\beno
\left\{
\begin{array}{l}
\Delta p_{\vupm_1, \vupm_2}= -\mathrm{tr}(\nabla\vu_1^\pm\nabla\vu_2^\pm)
\quad \text{in}\quad\Omega^\pm_f,\\
p_{\vupm_1, \vupm_2}=0\quad\text{on}\quad\Gamma_f,\\
\ve_3\cdot\nabla p_{\vupm_1, \vupm_2}=0\quad\text{on}\quad\Gamma^\pm.
\end{array}\right.
\eeno
Then for the pressure $p^\pm$, we have the following important representation
\begin{align}
p^\pm=\mathcal{H}^\pm\underline{p}^\pm+p_{\vupm, \vupm}-p_{\vhpm, \vhpm}.
\end{align}
Thus, we infer from (\ref{eq:theta-d}) that on $\Gamma_f$, we have
\begin{align*}
&\vN_f\cdot\nabla \mathcal{H}^+_f\underline{p}^+
-\vN_f\cdot\nabla \mathcal{H}^-_f\underline{p}^-\\
&=\Big[-2(u^+_1\partial_1\theta+u^+_2\partial_2\theta)-\vN_f\cdot\nabla(p_{\vup, \vup}-p_{\vhp, \vhp})
-\sum_{i,j=1,2}(u_i^+u^+_j-h^+_ih^+_j)\partial_i\partial_jf\Big]\\
&\quad+\Big[2(u^-_1\partial_1\theta+u^-_2\partial_2\theta)+\vN_f\cdot\nabla(p_{\vum, \vum}-p_{\vhm, \vhm})
+\sum_{i,j=1,2}(u^-_iu^-_j-h^-_ih^-_j)\partial_i\partial_jf\Big]\\
&\triangleq -g^++g^-.
\end{align*}
Thanks to the definition of DN operator, we get
\begin{align}
-\mathcal{N}^+_f\underline{p}^+-\mathcal{N}^-_f\underline{p}^-=-g^++g^-.
\end{align}
Recalling that
\begin{align*}
\underline{p}^+-\underline{p}^-=0\quad \text{ on }\quad\Gamma_f,\quad \widetilde{\mathcal{N}}_f={\mathcal{N}}_f^++{\mathcal{N}}_f^-,
\end{align*}
we get
\begin{align*}
\underline{p}^+=\underline{p}^-=\widetilde{\mathcal{N}}^{-1}_f(g^+-g^-).
\end{align*}
Therefore, from the fact that
\begin{align*}
\mathcal{N}^+_f\widetilde{\mathcal{N}}^{-1}_f g^-+\mathcal{N}^-_f\widetilde{\mathcal{N}}^{-1}_fg^+
=\frac12(g^++g^-)-\frac12(\mathcal{N}^+_f-\mathcal{N}^-_f)\widetilde{\mathcal{N}}^{-1}_f ( g^+-g^-),
\end{align*}
we obtain
\begin{align}
\partial_t\theta
=&\mathcal{N}^+_fp^+-g^+=\mathcal{N}^+_f\widetilde{\mathcal{N}}^{-1}_f(g^+-g^-)-g^+\nonumber\\
=&-\mathcal{N}^+_f\widetilde{\mathcal{N}}^{-1}_f g^--\mathcal{N}^-_f\widetilde{\mathcal{N}}^{-1}_f g^+\nonumber\\\nonumber
=&-
\big((\underline{u}^+_1+\underline{u}^-_1)\partial_1\theta+(\underline{u}^+_2+\underline{u}^-_2)\partial_2\theta\big)
\\\nonumber
&-\frac12\sum_{i,j=1,2}\big(\underline{u}_i^+\underline{u}^+_j-\underline{h}^+_i\underline{h}^+_j\big)\partial_i\partial_jf
-\frac12\sum_{i,j=1,2}\big(\underline{u}^-_i\underline{u}^-_j-\underline{h}^-_i\underline{h}^-_j\big)\partial_i\partial_jf\\\nonumber
&+\frac12(\mathcal{N}^+_f-\mathcal{N}^-_f)\widetilde{\mathcal{N}}^{-1}_f
\sum_{i,j=1,2}\big(\underline{u}_i^+\underline{u}^+_j-\underline{h}^+_i\underline{h}^+_j-\underline{u}^-_i\underline{u}^-_j
+\underline{h}^-_i\underline{h}^-_j\big)\partial_i\partial_jf
\\\nonumber
&+(\mathcal{N}^+_f-\mathcal{N}^-_f)\widetilde{\mathcal{N}}^{-1}_f
((\underline{u}^+_1-\underline{u}^-_1)\partial_1\theta+(\underline{u}^+_2-\underline{u}^-_2)\partial_2\theta)\\
&-\mathcal{N}^+_f\widetilde{\mathcal{N}}^{-1}_f(\vN\cdot\underline{\nabla(p_{\vum,\vum}-p_{\vhm,\vhm})})
-\mathcal{N}^-_f\widetilde{\mathcal{N}}^{-1}_f(\vN\cdot\underline{\nabla(p_{\vup, \vup}-p_{\vhp, \vhp})}).\label{eq:theta}
\end{align}

\begin{remark}
From $\div\vu^\pm=0$ and $u_3^\pm=0$ on $\Gamma^\pm$, we know that
\beno
\int_\BT\theta dx'=0,\quad\text{thus }\int_\BT\partial_t\theta dx'=0.
\eeno
Thanks to $\partial_t\theta=\mathcal{N}^+_fp^+-g^+=\mathcal{N}^-_fp^-+g^-$, we get
$$\int_{\T^2}g^\pm dx'=0.$$
Therefore, $\widetilde{\mathcal{N}}^{-1}_fg^\pm$ is well-defined.

However, we do not know whether the functions $$G_1=\sum_{i,j=1,2}(\underline{u}_i^+\underline{u}^+_j-\underline{h}^+_i\underline{h}^+_j-\underline{u}^-_i\underline{u}^-_j+\underline{h}^-_i\underline{h}^-_j)\partial_i\partial_jf$$
as well as
$$
G_2=(\underline{u}^+_1-\underline{u}^-_1)\partial_1\theta+(\underline{u}^+_2-\underline{u}^-_2)\partial_2\theta
$$
have zero integral on $\BT$. So, it is unreasonable to directly apply $\widetilde{\mathcal{N}}^{-1}_f$ to it.
Thus, we make the convention: if $\int_{\BT} h dx'\neq0$, we interpret $\cN_f^{-1}h$ as
$\cN_f^{-1}(h-\int_{\BT} h dx')$.
Note that this does not change the formulation of the system, since
\beno
\int_{\T^2}\f12G_1+G_2 dx'=0.
\eeno
\end{remark}

\subsection{The equations for the vorticity and current}
We denote
\begin{align*}
\vom^\pm\triangleq\nabla\times\vu^\pm,\quad \vj^\pm\triangleq\nabla\times\vh^\pm.
\end{align*}
Then in $\Om_f^\pm$, $(\vom^\pm, \vj^\pm)$ satisfies
\begin{align}\label{eq:vorticity-w}
&\partial_t\vom^\pm+\vu^\pm\cdot\nabla\vom^\pm-\vh^\pm\cdot\nabla\vj^\pm=\vom^\pm\cdot\nabla\vu^\pm-\vj^\pm\cdot\nabla\vh^\pm,\\
&\partial_t\vj^\pm+\vu^\pm\cdot\nabla\vj^\pm-\vh^\pm\cdot\nabla\vom^\pm=\vj^\pm\cdot\nabla\vu^\pm
-\vom^\pm\cdot\nabla\vh^\pm-2\nabla u_i^\pm\times\nabla h_i^\pm,\label{eq:vorticity-h}
\end{align}
where we used the fact that
\begin{align*}
&\veps^{ijk}\partial_ju_l\partial_lh_k-\veps^{ijk}\partial_jh_l\partial_lu_k\\
&=\veps^{ijk}\partial_ju_l(\partial_lh_k-\partial_kh_l)+\veps^{ijk}\partial_ju_l\partial_kh_l-\veps^{ijk}\partial_jh_l\partial_ku_l
+\veps^{ijk}\partial_jh_l(\partial_ku_l-\partial_lu_k)\\
&=-\vj\cdot\nabla\vu+\vom\cdot\nabla\vh-2\nabla u_i\times\nabla h_i.
\end{align*}

\subsection{Tangential velocity on $\Gamma^\pm$}
Let us derive the evolution equations for the following quantities
\begin{align*}
\beta_i^\pm(t)=\int_{\BT}u_i^\pm(t,x',\pm1)dx',\quad \gamma_i^\pm(t)=\int_{\BT}h_i^\pm(t,x',\pm1)dx'.
\end{align*}
The motivation is that we have to recover a vector field by its curl,
divergence, and normal components on upper and bottom boundary in the domain $\Omega_f^\pm$.
However, the solution may be not unique unless the mean values of their tangential components are given on top and bottom boundary.

Thanks to $u_3(t,x',\pm1)\equiv 0$, we deduce that for $i=1,2$
\begin{align*}
  \partial_tu_i+u_{j}\partial_{j}u_i-h_{j}\partial_{j}h_i-\partial_ip=0\quad \text{on}\quad \Gamma_\pm,
\end{align*}
which gives
\begin{align*}
  \partial_t\beta_i^\pm+\int_{\Gamma_\pm}\big(u_{j}\partial_{j}u_i-h_{j}\partial_{j}h_i\big)dx'=0,
\end{align*}
or equivalently
\begin{align}
\beta_i^\pm(t)=\beta_i^\pm(0)-\int_0^t\int_{\Gamma_\pm}\big(u_{j}\partial_{j}u_i-h_{j}\partial_{j}h_i\big)dx'dt.
\end{align}
Similarly, we have
\begin{align}
\gamma_i^\pm(t)=\gamma_i^\pm(0)-\int_0^t\int_{\Gamma_\pm}\big(u_{j}\partial_{j}h_i-h_{j}\partial_{j}u_i\big)dx'dt.
\end{align}

\subsection{Solvability conditions of div-curl system}
We have to recover the divergence-free velocity(magnetic) field from its curl part. Namely, we need to solve
the following div-curl system with certain boundary conditions:
\begin{align}\label{div-curl-temp}
\left\{
\begin{array}{l}
\curl \vu^\pm=\vom^\pm,\quad \div\vu^\pm=0\quad\text{in}\quad \Omega_f^\pm,\\
\vu^\pm\cdot\vN_f =\partial_tf\quad\text{on}\quad\Gamma_{f}, \\
\vu^\pm\cdot\ve_3 = 0, \quad \int_{\Gamma^\pm} u_i^\pm dx'=\beta^\pm_i (i=1,2)\quad \text{on}\quad\Gamma^{\pm}.
\end{array}\right.
\end{align}

The solvability of $\vu$ requires that $\vom$ must satisfy the following two compatibility conditions:
\begin{itemize}
\item[C1.]  $\div\vom^\pm=0$\, in $\Omega_f^\pm$

\item[C2.]  $\int_{\Gamma^\pm}\om_3^\pm dx'=0$,
\end{itemize}
and $\partial_t f$ must be average free, i.e.,
\begin{itemize}
\item[C3.]  $\int_{\BT}\partial_tf dx'=0$.
\end{itemize}

\section{Weak stability condition and linear stability}

\subsection{Weak stability condition}

Recall that the weak stability condition
\begin{align*}
\text{(S1)}\qquad \big|[\vu]\times\vh^+\big|^2+\big|[\vu]\times\vh^-\big|^2 < 2\big|\vh^+\times\vh^-\big|^2\quad \text{on}\quad\Gamma_f.
\end{align*}
It is not easy to see why the weak stability condition (S1) will ensure that the system (\ref{cvs})-(\ref{cvsi}) is well-posed.
So, we reformulate it into a new formulation.

The condition (S1) implies that
\begin{align*}
(\vh^+\cdot\mathbf{e}_1)^2+(\vh^-\cdot\mathbf{e}_1)^2>0,\quad (\vh^+\cdot\mathbf{e}_2)^2+(\vh^-\cdot\mathbf{e}_2)^2>0,
\end{align*}
where $\mathbf{e}_1=(1,0,0)$ and $\mathbf{e}_2=(0,1,0)$. Thanks to $\vu^+\cdot\vN_f=\vu^-\cdot\vN_f$, we know that $[\vu]\cdot\vN_f=0$. Thus, we may assume
\begin{align*}
[\vu]=\nu_1\vh^++\nu_2\vh^-.
\end{align*}
Then (S1) is equivalent to $\nu_1^2+\nu_2^2<2$, which is actually equivalent to
\begin{align*}
\text{(S2)}\qquad  \inf_{\vq\in T_{\Gamma_f}(x),|\vq|=1} 2\big[(\vh^+(x)\cdot\vq)^2+(\vh^-(x)\cdot\vq)^2\big]-([\vu](x)\cdot\vq)^2 >0,
\end{align*}
by simply using Cauchy-Schwartz inequality.
\medskip

In our graph case, we find that (S2) implies the following lemma.

\begin{lemma}\label{lem:stability}
There exists $c_0>0$ such that
\begin{align}\nonumber
\Lambda(\vh^\pm,[\vu])&\eqdefa\sup_{x\in\Gamma_f}\sup_{\varphi_1^2+\varphi_2^2=1}(h_1^+\varphi_1+h_2^+\varphi_2)^2+(h_1^-\varphi_1+h_2^-\varphi_2)^2
-\frac12(v_1\varphi_1+v_2\varphi_2)^2\\
&\ge c_0,\nonumber
\end{align}
where $[\vu]=2(v_1,v_2,v_3)$.
\end{lemma}
\begin{proof}
Let $\vq=(q_1, q_2, q_3)\bot\vN_f$ with $q_3=q_1\partial_1f+q_2\partial_2f$ and $(q_1,q_2)$ determined by
\begin{equation*}
\left(
  \begin{array}{cc}
    1+(\partial_1f)^2 & \partial_1f\partial_2f \\
    \partial_1f\partial_2f & 1+(\partial_1f)^2  \\
  \end{array}
\right)
\left(
         \begin{array}{c}
           q_1 \\
           q_2 \\
         \end{array}
\right)
= \left(
         \begin{array}{c}
           \varphi_1 \\
           \varphi_2 \\
         \end{array}
\right).
\end{equation*}
Then with the fact $\vh^\pm\cdot\vN_f=0$, we have
\beno
h_1^\pm\varphi_1+h_2^\pm\varphi_2=\sum_{i=1}^3h_i^\pm q_i.
\eeno
Similarly, we have
\beno
w_1\varphi_1+w_2\varphi_2=\sum_{i=1}^3w_i q_i.
\eeno
Thus, (S2) tells us that
\begin{align*}
\inf_{\varphi_1^2+\varphi_2^2=1} 2[(h_1^+\varphi_1+h_2^+\varphi_2)^2+(h_1^-\varphi_1+h_2^-\varphi_2)^2]
-(w_1\varphi_1+w_2\varphi_2)^2>0.
\end{align*}
The above inequality holds for all $x\in\Gamma_f$. Thus, there exists a constant $c_0>0$ such that
\begin{align*}
\inf_{\varphi_1^2+\varphi_2^2=1} 2[(h_1^+\varphi_1+h_2^+\varphi_2)^2+(h_1^-\varphi_1+h_2^-\varphi_2)^2]
-(w_1\varphi_1+w_2\varphi_2)^2 \ge c_0,
\end{align*}
which gives rise to the lemma.
\end{proof}

\subsection{Linear stability}
Let $(\vu^\pm,\vh^\pm)$ be a constant solution of the system (\ref{cvs}).
The linearized system of (\ref{eq:form:f}) and (\ref{eq:theta}) around $(\vu^\pm,\vh^\pm)$ takes as follows
\begin{align*}
\begin{array}{l}
\partial_tf=\theta,\\
\partial_t\theta=\mathcal{A}f+L(\theta,f),
\end{array}
\end{align*}
where $\mathcal{A}$ is a linear operator of second order defined by
\beno
\mathcal{A}f=-2(w_1\partial_1\theta+w_2\partial_2\theta)+\sum_{i,j=1,2}(-w_iw_j-v_iv_j+\f12h^+_ih^+_j+\f12h^-_ih^-_j)\partial_i\partial_jf,
\eeno
with $(w_1, w_2, w_3)=\f12(\vu^++\vu^-)$ and $(v_1,v_2,v_3)=\f12(\vu^+-\vu^-)=\f12[\vu]$, and $L(\theta,f)$ denotes the lower order linear terms.

It is easy to verify that
\begin{align*}
(\partial_t+w_1\partial_1+w_2\partial_2)^2f=\f12\sum_{i,j=1,2}(-2v_iv_j+h^+_ih^+_j+h^-_ih^-_j)\partial_i\partial_jf+L(\theta,f).
\end{align*}
Let $D_t=\partial_t+w_1\partial_1+w_2\partial_2$, then we get
\begin{align*}
D_t^2f=\f12\sum_{i=1,2}(-2v_iv_j+h^+_ih^+_j+h^-_ih^-_j)\partial_i\partial_jf+L(\theta,f).
\end{align*}
The principal symbol of the operator on the right hand side is
\begin{align}
(v_i\xi_i)^2-\f12\big((h^+_i\xi_i)^2+(h^-_i\xi_i)^2\big),
\end{align}
which is negative by weak stability condition (S2). This means that $f$ satisfies a strictly
hyperbolic equation. Thus, the system is linearly well-posed.

\section{Uniform estimates for the linearized system}

Given ${f}(t,x'),{\vu}^{\pm}(t,x),{\vh}^{\pm}(t,x)$, we assume that there exists $T>0$ and positive constant $L_0, L_1, L_3$ such that for any $t\in [0,T]$, there holds
\begin{align}
&\|(\vu^\pm, \vh^\pm)(t)\|_{L^{\infty}(\Gamma_f)}\le L_0,\\
&\|f(t)\|_{H^{s+\f12}(\BT)}+\|\pa_tf(t)\|_{H^{s-\f12}(\BT)}+\|\vu^\pm(t)\|_{H^{s}(\Omega_f^\pm)}
+\|\vh^\pm(t)\|_{H^{s}(\Omega_f^\pm)}\le L_1,\\
&\|(\partial_t\vu^\pm, \partial_t\vh^\pm)(t)\|_{L^{\infty}(\Gamma_f)}\le L_2,\\
&\|f(t)-f_*\|_{H^{s-\f12}}\le \delta_0,\\
&-(1-c_0)\le f(t,x')\le (1-c_0),\\
&\Lambda(\vh^\pm, [\vu])(t)\ge c_0,\label{ass:stability}
\end{align}
together with
\ben\label{ass:boun}
\left\{
\begin{array}{l}
\div\vu^\pm=\div\vh^\pm=0\quad\text{in}\quad\Omega_f^\pm,\\
\underline{\vh}^\pm\cdot\vN_f=0, \quad \partial_tf=\underline{\vu}^\pm\cdot\vN_f,\\
u_3^\pm=h_3^\pm=0\quad\text{in}\quad \Gamma^\pm.
\end{array}\right.
\een

In this section, we linearize the equivalent system derived in section 5 around $(f,\vu^\pm,\vh^\pm)$, and present the uniform energy estimates
for the linearized system.

\subsection{The linearized system of $(f,\theta)$}
For the system (\ref{eq:form:f}) and (\ref{eq:theta}), we introduce the following linearized system:
\begin{equation}\label{sys:linear-H}
\left\{\begin{split}
\partial_t\bar{f}=&~\bar{\theta},\\
\partial_t\bar{\theta}=&-\left((\underline{u}^+_1+\underline{u}^-_1)
\partial_1\bar{\theta}+(\underline{u}^+_2+\underline{u}^-_2)\partial_2\bar\theta\right)\\
&~-\frac12\sum_{i,j=1,2}\big(\underline{u}_i^+\underline{u}^+_j-\underline{h}^+_i\underline{h}^+_j+\underline{u}^-_i\underline{u}^-_j-\underline{h}^-_i\underline{h}^-_j\big)
\partial_i\partial_j\bar{f}+\mathfrak{g},
\end{split}\right.
\end{equation}
where
\begin{align}
\mathfrak{g}=&~\frac12(\mathcal{N}_f^+-\mathcal{N}_f^-)\bar{\mathcal{N}}_f^{-1}
\sum_{i,j=1,2}\big(\underline{u}_i^+\underline{u}^+_j-\underline{h}^+_i\underline{h}^+_j-\underline{u}^-_i\underline{u}^-_j+\underline{h}^-_i\underline{h}^-_j\big)\partial_i\partial_jf\nonumber\\
&~+(\mathcal{N}_f^+-\mathcal{N}_f^-)\bar{\mathcal{N}}_f^{-1}
\left((\underline{u}^+_1-\underline{u}^-_1)\partial_1\theta+(\underline{u}^+_2-\underline{u}^-_2)\partial_2\theta\right)\nonumber\\
&~-\mathcal{N}_f^+\bar{\mathcal{N}}_f^{-1}\left(\vN_f\cdot\underline{\nabla(p_{\vum,\vum}-p_{\vhm,\vhm})}\right)\nonumber\\
&~-\mathcal{N}_f^-\bar{\mathcal{N}}_f^{-1}\left(\vN_f\cdot\underline{\nabla(p_{\vup, \vup}-p_{\vhp, \vhp})}\right)\nonumber\\
\triangleq& \mathfrak{g}_1+\cdots+\mathfrak{g}_4.\label{eq:g-def}
\end{align}
 Recall that $\cN_f^{-1}h$ is interpreted as $\cN_f^{-1}(h-\int_{\BT} h dx')$ if $\int_{\BT} h dx'\neq0.$
 We remark that $\int_{\BT}\bar\theta dx'$ may not vanish since we have performed the linearization.\medskip

Let us introduce the energy functional $E_s$ defined by
\begin{align}\nonumber
E_s(\partial_t\bar{f},\bar{f})\eqdefa &\big\|(\partial_t+w_i\partial_i)\Ds\bar{f}\big\|_{L^2}^2
-\big\|v_i\partial_i\Ds\bar{f}\big\|_{L^2}^2\\
&+\frac12\big\|\underline{h}^+_i\partial_i\Ds\bar{f}\big\|_{L^2}^2
+\frac12\big\|\underline{h}^-_i\partial_i\Ds\bar{f}\big\|_{L^2}^2,
\end{align}
where $\langle \na\rangle^s f=\mathcal{F}^{-1}((1+|\xi|^2)^{\f s 2}\widehat{f})$ and
\begin{align*}
w_i=\frac12(\underline{u}^+_i+\underline{u}^-_i),\quad v_i=\frac12(\underline{u}^+_i-\underline{u}^-_i).
\end{align*}
It is easy to see that there exists $C(L_0)>0$ so that
\begin{align}\label{linear:equi-norm-1}
E_s(\partial_t\bar{f},\bar{f})\le C(L_0)\left(\|\partial_t\bar{f}\|_{H^{s-\f12}}^2
+\|\bar{f}\|_{H^{s+\f12}}^2\right).
\end{align}
The stability condition (\ref{ass:stability}) insures that there exists $C(c_0,L_0)$ so that
\begin{align}\label{linear:equi-norm-2}
&\|\partial_t \bar{f}\|_{H^{s-\f12}}^2+\|\bar{f}\|_{H^{s+\f12}}^2
\le C(c_0,L_0)\Big\{E_s(\partial_t\bar{f},\bar{f})+\|\partial_t\bar{f}\|_{L^2}^2+\|\bar{f}\|_{L^2}^2\Big\}.
\end{align}

\begin{proposition}\label{prop:f-L}
Assume that $\mathfrak{g}\in L^\infty(0,T;H^{s-\f12}(\T^2))$. Given the initial data $(\bar \theta_0, \bar f_0)\in H^{s-\f12}\times H^{s+\f12}(\T^2)$, there exists a unique solution $(\bar f,\bar \theta)\in C\big([0,T];H^{s+\f12}\times H^{s-\f12}(\T^2)\big)$ to the system (\ref{sys:linear-H}) so that
\begin{align*}
&\sup_{t\in[0,T]}\left(\|\partial_t\bar{f}(t)\|_{H^{s-\f12}}^2+\|\bar{f}(t)\|_{H^{s+\f12}}^2\right)\\
 &\quad\le C(c_0,L_0)\left(\|\bar\theta_0\|_{H^{s-\f12}}^2+\|\bar f_0\|_{H^{s+\f12}}^2+\int_0^T\|\mathfrak{g}(\tau)\|_{H^{s-\f12}}d\tau\right)e^{C(c_0, L_1,L_2)T}.
 \end{align*}
\end{proposition}

\begin{proof}
We only present the uniform estimates, which ensure the existence and uniqueness of the solution.
Using the fact that
\begin{align*}
\pa_t^2\bar f=-2\sum_{i=1,2}w_i\pa_i\pa_t\bar f+\f12\sum_{i,j=1,2}(-2w_iw_j-2v_iv_j+\underline{h}^+_i\underline{h}^+_j+\underline{h}^-_i\underline{h}^-_j)\partial_i\partial_j\bar f+\mathfrak{g},
\end{align*}
a direct calculation shows that
\begin{align*}
&\frac12\frac{d}{dt}\big\|(\partial_t+w_i\partial_i)\Ds\bar{f}\big\|_{L^2(\BT)}^2\\
&= \Big\langle(\partial_t+w_i\partial_i)\Ds\bar{f}, \Ds\partial_{t}^2\bar{f}+w_i\partial_i(\Ds\partial_t\bar{f})
+\partial_tw_i\partial_i\Ds\bar{f}\Big\rangle\\
&=\Big\langle(\partial_t+w_i\partial_i)\Ds\bar{f}, \Ds\big(-2w_i\partial_i\partial_t\bar{f}
-(w_iw_j+v_iv_j)\partial_i\partial_j\bar{f}+\frac12(\underline{h}^+_i\underline{h}^+_j+\underline{h}^-_i\underline{h}^-_j)\partial_i\partial_j\bar{f}\big)\Big\rangle\nonumber\\
&\quad +\Big\langle(\partial_t+w_i\partial_i)\Ds\bar{f}, \Ds
\mathfrak{g}+w_i\partial_i(\Ds\partial_t\bar{f})
+\partial_tw_i\partial_i\Ds\bar{f}\Big\rangle\\
&=\Big\langle(\partial_t+w_i\partial_i)\Ds\bar{f}, -w_i\partial_i\Ds\partial_t\bar{f}\Big\rangle\\
&\quad+\Big\langle(\partial_t+w_i\partial_i)\Ds\bar{f},-(w_iw_j+v_iv_j)\partial_i\partial_j\Ds\bar{f}+\frac12(\underline{h}^+_i\underline{h}^+_j+\underline{h}^-_i\underline{h}^-_j)
\partial_i\partial_j\Ds\bar{f})\Big\rangle\nonumber\\
&\quad+\Big\langle(\partial_t+w_i\partial_i)\Ds\bar{f},\big[w_i, \Ds\big]\partial_i\partial_t\bar{f}\Big\rangle\nonumber\\
&\quad+\Big\langle(\partial_t+w_i\partial_i)\Ds\bar{f},\big[w_iw_j+v_iv_j-\frac12\underline{h}^+_i\underline{h}^+_j-\frac12\underline{h}^-_i\underline{h}^-_j,\Ds\big]\partial_i\partial_j\bar{f})\Big\rangle\nonumber\\
&\quad +\Big\langle(\partial_t+w_i\partial_i)\Ds\bar{f},\Ds
\mathfrak{g}+\partial_tw_i\partial_i\Ds\bar{f}\Big\rangle\\
&\triangleq I_1+\cdots I_5.
\end{align*}

It follows from Lemma \ref{lem:commutator} that
\begin{align*}
I_3\le& 2\|(\partial_t+w_i\partial_i)\Ds\bar{f}\|_{L^2}\big\|\big[w_i, \Ds\big]\partial_i\partial_t\bar{f}\big\|_{L^2} \\
\le& CE_s(\partial_t\bar{f},\bar{f})^\f12\|w\|_{H^{s-\f12}}\|\pa_t\bar f\|_{H^{s-\f12}}
\end{align*}
as well as
\begin{align*}
I_4\le CE_s(\partial_t\bar{f},\bar{f})^\f12\Big(\|w\|_{H^{s-\f12}}^2+\|v\|_{H^{s-\f12}}^2+\|\underline{\vh}^\pm\|_{H^{s-\f12}}^2\Big)\|\bar f\|_{H^{s+\f12}}.
\end{align*}
Obviously, it holds that
\beno
I_5\le E_s(\partial_t\bar{f},\bar{f})^\f12\big(\|\mathfrak{g}\|_{H^{s-\f12}}+\|\partial_tw\|_{L^\infty}\|\bar f\|_{H^{s+\f12}}\big).
\eeno

We get by integration by parts that
\begin{align*}
&\Big\langle\partial_t\Ds\bar{f}, ~-w_i\partial_i\Ds\partial_t\bar{f}\Big\rangle
\le \|\partial_iw_i\|_{L^\infty}\|\partial_t\bar{f}\|_{H^{s-\f12}}^2,\\
&\Big\langle w_i\partial_i\Ds\bar{f}, ~-w_i\partial_i\Ds\partial_t\bar{f}\Big\rangle
+\frac12\frac{d}{dt}\|w_i\partial_i\Ds\bar{f}\|_{L^2}^2\\
&\quad\qquad=\Big\langle w_i\partial_i \Ds\bar{f},\partial_tw_i\partial_i\Ds\bar{f}\Big\rangle\le
\|w\|_{L^\infty} \|\partial_tw\|_{L^\infty}\|\bar{f}\|_{H^{s+\f12}}^2,
\end{align*}
which give rise to
\beno
I_1\le -\frac12\frac{d}{dt}\|w_i\partial_i\Ds\bar{f}\|_{L^2}^2+\big(1+\|w\|_{W^{1,\infty}}+\|\partial_tw\|_{L^\infty}\big)^2
\Big(\|\bar{f}\|_{H^{s+\f12}}^2+\|\partial_t\bar{f}\|_{H^{s-\f12}}^2\Big).
\eeno

Similarly, we have
\begin{align*}
&\Big\langle\partial_t\Ds\bar{f}, -w_iw_j\partial_i\partial_j\Ds\bar{f}\Big\rangle
-\frac12\frac{d}{dt}\|w_i\partial_i\Ds\bar{f}\|_{L^2}^2\\
&=-\Big\langle w_i\partial_i\Ds\bar{f}, ~\partial_tw_i
\partial_i\Ds\bar{f}\Big\rangle+\Big\langle\Ds\partial_t\bar{f}, ~\partial_i(w_iw_j)\partial_j\Ds\bar{f}\Big\rangle\\
&\le\|w\|_{L^\infty}\big(\|\partial_tw\|_{L^\infty}+\|\nabla w\|_{L^\infty}\big)
\Big(\|\bar{f}\|_{H^{s+\f12}}^2+\|\partial_t\bar{f}\|_{H^{s-\f12}}^2\Big),\\
&\Big\langle w_k\partial_k\Ds\bar{f}, -w_iw_j\partial_i\partial_j\Ds\bar{f}\Big\rangle\\
&=\Big\langle \partial_i(w_kw_iw_j)\partial_k\Ds\bar{f}, \partial_j\Ds\bar{f}\Big\rangle
-\Big\langle w_kw_iw_j\partial_k\partial_i\Ds\bar{f},\partial_j\Ds\bar{f}\Big\rangle\\
&=\Big\langle \partial_i(w_kw_iw_j)\partial_k\Ds\bar{f}, \partial_j\Ds\bar{f}\Big\rangle
-\Big\langle w_k\partial_k(w_i\partial_i\Ds\bar{f}),w_j\partial_j\Ds\bar{f}\Big\rangle\\
&\quad+\Big\langle w_k(\partial_kw_i)\partial_i\Ds\bar{f}), w_j\partial_j\Ds\bar{f}\Big\rangle\\
&\le C \|w\|^2_{L^\infty}\|\nabla w\|_{L^\infty}\|\bar{f}\|_{H^{s+\f12}}^2,
\end{align*}
as well as
\begin{align*}
&\Big\langle\partial_t\Ds\bar{f}, -v_iv_j\partial_i\partial_j\Ds\bar{f}\Big\rangle\\
&\le\frac12\frac{d}{dt}\|v_i\partial_i\Ds\bar{f}\|_{L^2}^2+\|v\|_{L^\infty}
\big(\|\partial_tv\|_{L^\infty}+\|\nabla v\|_{L^\infty}\big)
\Big(\|\bar{f}\|_{H^{s+\f12}}^2+\|\partial_t\bar{f}\|_{H^{s-\f12}}^2\Big),\\
&\Big\langle\partial_t\Ds\bar{f}, \underline{h}^\pm_i\underline{h}^\pm_j\partial_i\partial_j\Ds\bar{f}\Big\rangle\\
&\le-\frac12\frac{d}{dt}\|\underline{h}^\pm_i\partial_i\Ds\bar{f}\|_{L^2}^2+\|\underline{\vh}^\pm\|_{L^\infty}
\big(\|\partial_t\underline{\vh}^\pm\|_{L^\infty}+\|\nabla \underline{\vh}^\pm\|_{L^\infty}\big)
\Big(\|\bar{f}\|_{H^{s+\f12}}^2+\|\partial_t\bar{f}\|_{H^{s-\f12}}^2\Big).
\end{align*}
Thus, we obtain
\begin{align*}
I_2\le& \frac12\frac{d}{dt}\|w_i\partial_i\Ds\bar{f}\|_{L^2}^2+\frac12\frac{d}{dt}\|v_i\partial_i\Ds\bar{f}\|_{L^2}^2\\
&-\frac14\frac{d}{dt}\|\underline{h}^+_i\partial_i\Ds\bar{f}\|_{L^2}^2-\frac14\frac{d}{dt}\|\underline{h}^-_i\partial_i\Ds\bar{f}\|_{L^2}^2\\
&+C\big(1+\|(\underline{\vu}^\pm,\underline{\vh}^\pm)\|_{W^{1,\infty}}+\|\partial_t(\underline{\vu}^\pm, \underline{\vh}^\pm)\|_{L^\infty}\big)^3\Big(\|\bar{f}\|_{H^{s+\f12}}^2+\|\partial_t\bar{f}\|_{H^{s-\f12}}^2\Big).
\end{align*}

Putting the estimates of $I_1,\cdots, I_5$ together, we conclude that
\begin{align*}
\frac{d}{dt}E_s(&\partial_t\bar{f},\bar{f})\le \|\mathfrak{g}\|_{H^{s-\f12}}^2\\ +&C(L_0)\big(1+\|(\underline{\vu}^\pm,\underline{\vh}^\pm)\|_{H^{s-\f12}}+\|\partial_t(\underline{\vu}^\pm,\underline{\vh}^\pm)\|_{L^\infty}\big)^3\Big(\|\bar{f}\|_{H^{s+\f12}}^2+\|\partial_t\bar{f}\|_{H^{s-\f12}}^2\Big).
\end{align*}
On the other hand, it is easy to show that
\begin{align*}
\frac{d}{dt}\big(\|\partial_t\bar{f}\|_{L^2}^2+\|\bar{f}\|_{L^2}^2\big)\le C(L_0)\Big(\|\bar{f}\|_{H^{s+\f12}}^2+\|\partial_t\bar{f}\|_{H^{s-\f12}}^2\Big)
+\|\mathfrak{g}\|_{L^2}^2.
\end{align*}

Let $\mathcal{E}(t)\triangleq\|\bar{f}(t)\|_{H^{s+\f12}}^2+\|\partial_t\bar{f}(t)\|_{H^{s-\f12}}^2$.
Then we get by (\ref{linear:equi-norm-2}) that
\begin{align*}
\mathcal{E}(t)\le C(c_0,L_0)\Big(&\|\bar\theta_0\|_{H^{s-\f12}}^2+\|\bar f_0\|_{H^{s+\f12}}^2+\int_0^t\|\mathfrak{g}(\tau)\|_{H^{s-\f12}}^2d\tau\\
&+\int_0^t\big(1+\|(\underline{\vu}^\pm,\underline{\vh}^\pm)(\tau)\|_{H^{s-\f12}}+\|\partial_t(\underline{\vu}^\pm, \underline{\vh}^\pm)(\tau)\|_{L^\infty}\big)^3\mathcal{E}(\tau)d\tau\Big),
\end{align*}
which along with Lemma \ref{lem:basic} gives
\begin{align*}
\mathcal{E}(t)\le C(c_0,L_0)\Big(&\|\bar\theta_0\|_{H^{s-\f12}}^2+\|\bar f_0\|_{H^{s+\f12}}^2+\int_0^t\|\mathfrak{g}(\tau)\|_{H^{s-\f12}}^2d\tau+C(L_1,L_2)\int_0^t\mathcal{E}(\tau)d\tau\Big).
\end{align*}
This gives the desired estimate by Gronwall's inequality.
\end{proof}

Let us conclude this subsection by the estimate of $\mathfrak{g}$ defined by (\ref{eq:g-def}).

\begin{lemma}\label{lem:non-g}
It holds that
\begin{align*}
\|\mathfrak{g}\|_{H^{s-\f12}}\le C(L_1).
\end{align*}
\end{lemma}

\begin{proof}
It follows from Proposition \ref{prop:DN-Hs}, Proposition \ref{prop:DN-inverse} and Lemma \ref{lem:basic} that
\begin{align*}
\|\mathfrak{g}_1\|_{H^{s-\f12}}\le& C(L_1)\left\|(\underline{u}_i^+\underline{u}^+_j-\underline{h}^+_i\underline{h}^+_j-\underline{u}^-_i\underline{u}^-_j+\underline{h}^-_i\underline{h}^-_j)\partial_i\partial_jf
\right\|_{H^{s-\f32}}\\
\le& C(L_1)\|(\underline{\vu}^\pm,\underline{\vh}^\pm)\|_{H^{s-\f32}}\|f\|_{H^{s+\f12}}\le C(L_1)
\end{align*}
as well as
\begin{align*}
\|\mathfrak{g}_2\|_{H^{s-\f12}}
\le& C(L_1)\|\underline{\vu}^\pm\|_{H^{s-\f32}}\|\theta\|_{H^{s-\f12}}\le C(L_1),
\end{align*}
and by Proposition \ref{prop:elliptic-c},
\begin{align*}
\|(\mathfrak{g}_3, \mathfrak{g}_4)\|_{H^{s-\f12}}
\le& C(L_1)\|\underline{\nabla(p_{\vum,\vum}-p_{\vhm,\vhm})}\|_{H^{s-\f12}}+C(L_1)\|\underline{\nabla(p_{\vup, \vup}-p_{\vhp, \vhp})}\|_{H^{s-\f12}}\\
\le& C(L_1)\big\|\nabla(p_{\vum,\vum}, p_{\vhm,\vhm})\big\|_{H^{s}(\Om_f^-)}
+\big\|\nabla(p_{\vup, \vup}, p_{\vhp, \vhp})\big\|_{H^{s}(\Om_f^+)}\\
\le& C(L_1)\|(\vu^\pm,\vh^\pm)\|_{H^s(\Om_f^\pm)}\le C(L_1).
\end{align*}
The proof is finished.
\end{proof}

\subsection{The linearized system of $(\vw, \vj)$}

For the vorticity system (\ref{eq:vorticity-w})-(\ref{eq:vorticity-h}),
we introduce the following linearized system:
\begin{align}\label{eq:vorticity-w-L}
&\partial_t\bar\vom^\pm+\vu^\pm\cdot\nabla\bar\vom^\pm-\vh^\pm\cdot\nabla\bar\vj^\pm
=\bar\vom^\pm\cdot\nabla\vu^\pm-\bar\vj^\pm\cdot\nabla\vh^\pm,\\\label{eq:vorticity-h-L}
&\partial_t\bar\vj^\pm+\vu^\pm\cdot\nabla\bar\vj^\pm-\vh^\pm\cdot\nabla\bar\vom^\pm
=\bar\vj^\pm\cdot\nabla\vu^\pm-\bar\vom^\pm\cdot\nabla\vh^\pm-2\nabla u_i^\pm\times\nabla h_i^\pm.
\end{align}

Let $\vvap^\pm=\bar\vom^\pm+\bar\vj^\pm$, which satisfies
\begin{align}\label{eq:vorticity-d1}
\partial_t\vvap^\pm+(\vu-\vh)^\pm\cdot\nabla\vvap^\pm=\vvap^\pm\cdot\nabla(\vh-\vu)^\pm-2\nabla u_i^\pm\times\nabla h_i^\pm.
\end{align}
We define the flow map $X^\pm(t,\cdot)$
\begin{align*}
\frac{dX^\pm(t,x)}{dt}=(\vu-\vh)^\pm(t,X^\pm(t,x)),\quad x\in \Omega^\pm_{f_0}.
\end{align*}
Thanks to (\ref{ass:boun}), $X^\pm(t,\cdot)$ is a map from $\Omega^\pm_{f_0}$ to $\Omega^\pm_{f(t)}$.
Then we have
\begin{align*}
\frac{d\vvap^\pm(t, X^\pm(t,x))}{dt}=\Big(\vvap^\pm\cdot\nabla(\vh-\vu)^\pm-
2\nabla u_l^\pm\times\nabla h_l^\pm\Big)(t,X^\pm(t,x)),\quad x\in \Omega^\pm_0.
\end{align*}
This is a linear ODE system, which admits a unique solution apparently.
$\bar\vom^\pm-\bar\vj^\pm$ can be solved in the same way. Moreover, we have the following estimate.

\begin{proposition}\label{prop:vorticity}
It holds that
\begin{align*}
\sup_{t\in[0,T]}&\Big(\|\bar\vom^\pm(t)\|^2_{H^{s-1}(\Omega_f^\pm)}+\|\bar\vj^\pm(t)\|^2_{H^{s-1}(\Omega_f^\pm)}\Big)\\
 &\qquad\le \Big(1+\|\bar\vom^\pm_0\|_{H^{s-1}(\Omega_0^\pm)}^2+\|\bar\vj^\pm_0\|_{H^{s-1}(\Omega_0^\pm)}^2\Big)e^{C(L_1)T}.
 \end{align*}
\end{proposition}

\begin{proof}
Thanks to $\partial_tf=\underline{\vu}^\pm\cdot\vN_f$, we get by integration by parts that
\begin{align*}
&\frac12\frac{d}{dt}\int_{\Omega_f^\pm}|\nabla^{s-1}\vvap^\pm(t,x)|^2dx\\
&=\int_{\Omega_f^\pm}\nabla^{s-1}\vvap^\pm\cdot\nabla^{s-1}\partial_t\vvap^\pm dx+\frac12\int_{\Gamma_f}|\nabla^{s-1}\vvap^\pm|^2(\vu^\pm\cdot\vn)d\sigma.
\end{align*}
Using the equation (\ref{eq:vorticity-d1}) and $\underline{\vh}^\pm\cdot\vN_f=0$, we get by Lemma \ref{lem:basic} that
\begin{align*}
&\frac12\frac{d}{dt}\int_{\Omega_f^\pm}|\nabla^{s-1}\vvap^\pm(t,x)|^2dx\\
&\le\int_{\Omega_f^\pm}\nabla^{s-1}\vvap^\pm\cdot\nabla^{s-1}\big[(\vu-\vh)^\pm\cdot\nabla\vvap^\pm\big]dx
+\frac12\int_{\Gamma_f}|\nabla^{s-1}\vvap^\pm|^2(\vu^\pm\cdot\vn)d\sigma\\
&\qquad+C(L_1)\Big(1+\|\bar\vom^\pm(t)\|^2_{H^{s-1}(\Omega_f^\pm)}+\|\bar\vj^\pm(t)\|^2_{H^{s-1}(\Omega_f^\pm)}\Big)\\
&\le\frac12\int_{\Omega_f^\pm}(\vu-\vh)^\pm\cdot\nabla\big(|\nabla^{s-1}\vvap^\pm|^2\big)dx
+\frac12\int_{\Gamma_f}|\nabla^{s-1}\vvap^\pm|^2(\vu^\pm\cdot\vn)d\sigma\\
&\qquad+C(L_1)\Big(1+\|\bar\vom^\pm(t)\|^2_{H^{s-1}(\Omega_f^\pm)}+\|\bar\vj^\pm(t)\|^2_{H^{s-1}(\Omega_f^\pm)}\Big)\\
&=-\frac12\int_{\Omega_f^\pm}\div(\vu-\vh)^\pm|\nabla^{s-1}\vvap^\pm|^2dx\\
&\qquad+C(L_1)\Big(1+\|\bar\vom^\pm(t)\|^2_{H^{s-1}(\Omega_f^\pm)}+\|\bar\vj^\pm(t)\|^2_{H^{s-1}(\Omega_f^\pm)}\Big)\\
&\le C(L_1)\Big(1+\|\bar\vom^\pm(t)\|^2_{H^{s-1}(\Omega_f^\pm)}+\|\bar\vj^\pm(t)\|^2_{H^{s-1}(\Omega_f^\pm)}\Big).
\end{align*}

In a similar way, we can deduce that
\begin{align*}
\frac12\frac{d}{dt}\int_{\Omega_f^\pm}|\nabla^{s-1}(\bar\vom-\bar\vj)^\pm|^2dx
\le C(L_1)\Big(1+\|\bar\vom^\pm(t)\|^2_{H^{s-1}(\Omega_f^\pm)}+\|\bar\vj^\pm(t)\|^2_{H^{s-1}(\Omega_f^\pm)}\Big).
\end{align*}
Then the proposition follows from Gronwall's inequality.
\end{proof}

The solution of the system (\ref{eq:vorticity-w-L})-(\ref{eq:vorticity-h-L}) satisfies the following properties,
which are important compatibility conditions for solving the velocity and magnetic field from the vorticity and current.

\begin{lemma}\label{lem:com-vorticity}
It holds that
\beno
\frac{d}{dt}\int_{\Gamma^\pm}\bar\om^\pm_3dx'=0,\quad \frac{d}{dt}\int_{\Gamma^\pm}\bar\xi^\pm_3dx'=0.
\eeno
\end{lemma}
\begin{proof}Since $\partial_iu_3^\pm=\partial_ih_3^\pm=0(i=1,2)$ on $\Gamma^\pm$, we get
\begin{align*}
\frac{d}{dt}  \int_{\Gamma^+}\bar\om^+_3dx'
=&\int_{\Gamma^+}\big(-u^+_1\partial_1\bar\om_3^+
-u^+_2\partial_2\bar\om_3^++\bar{\om}^+_3\partial_3u_3^+\big)dx'\nonumber\\
&+\int_{\Gamma^+}\big(h^+_1\partial_1\bar\xi_3^+
+h^+_2\partial_2\bar\xi_3^+-\bar{\xi}^+_3\partial_3h_3^+\big)dx'\nonumber\\
=&\int_{\Gamma^+}\big(\partial_1u^+_1
+\partial_2u^+_2+\partial_3u_3^+\big)\bar{\om}_3dx'\nonumber\\
&-\int_{\Gamma^+}\big(\partial_1h^+_1+\partial_2h^+_2+\partial_3h_3^+\big)\bar\xi_3^+dx'=0.
\end{align*}
Similarly, we have
\begin{align*}\nonumber
\frac{d}{dt}  \int_{\Gamma^+}\bar\xi^+_3dx'
=&-2\int_{\Gamma^+}\big(\partial_1u^+_i\partial_2h_i^+-\partial_2u^+_i\partial_1h_i^+\big)dx'\\
=&2\int_{\Gamma^+}\big(u^+_i\partial_1\partial_2h_i^+-u^+_i\partial_2\partial_1h_i^+\big)dx'=0.
\end{align*}
The proof for $\om^-_3, \xi^-_3$ is similar.
\end{proof}

\section{Construction of the iteration map}

Assume that
\begin{align*}
f_0\in H^{s+\f12}(\BT),\quad  \vu_0^\pm,\,\vh_0^\pm\in H^{s}(\Omega_{f_0}^\pm).
\end{align*}
Furthermore, assume that there exists $c_0>0$ so that
\begin{itemize}
\item[1.] $-(1-2c_0)\le f_0(x')\le (1-2c_0)$;

\item[2.] $\Lambda(\vh_0^\pm,[\vu_0])\ge 2c_0$.
\end{itemize}

We choose $f_*=f_0$ and take $\Omega_*^\pm=\Omega_{f_0}^\pm$ as the reference region.
The initial data $(f_I,(\partial_tf)_I,\vom_{*I}^\pm,\vj_{*I}^\pm,\beta_{Ii}^\pm,
\Gamma_{Ii}^\pm)$ for the equivalent system is defined as follows
\begin{align*}
&f_I=f_0,\quad (\partial_tf)_I=\vu_0^\pm(x',f_0(x'))\cdot(-\partial_1f_0,-\partial_2f_0,1),\\
&\vom_{*I}^\pm=\curl\vu_0^\pm,\quad \vj_{*I}^\pm=\curl\vh_0^\pm,\\
&\beta_{Ii}^\pm=\int_{\BT}u_{0i}^\pm(x',\pm1)dx',\quad \gamma_{Ii}^\pm=\int_{\BT}h_{0i}^\pm(x',\pm1)dx'.
\end{align*}
In addition, we choose a large constant $M_0>0$ so that
\begin{align}
\|f_I\|_{H^{s+\f12}}+\|(\vom_{I*}^\pm, \vj_{I*}^\pm)\|_{H^{s-1}(\Omega_*^\pm)}
+\|(\partial_tf)_I\|_{H^{s-\f12}}+|\beta^\pm_{Ii}|+|\gamma^\pm_{Ii}| \le M_0.
\end{align}

To construct the iteration map, we introduce the following functional space.

\begin{definition}\label{def:X}
Given two positive constants $M_1, M_2>0$ with $M_1>2M_0$, we define
the space $\mathcal{X}=\mathcal{X}(T, M_1, M_2)$ as the collection of
$(f, \vom_*^\pm, \vj_*^\pm, \beta^\pm_{i},\gamma^\pm_{i})$, which satisfies
\begin{align*}
&\left(f(0),\partial_tf(0), \vom_*^\pm(0), \vj_*^\pm(0),\beta^\pm_{i}(0),\gamma^\pm_{i}(0)\right)=\big(f_I, (\partial_t f)_I, \vom_{*I}^\pm, \vj_{*I}^\pm, \beta^\pm_{iI},\gamma^\pm_{iI}\big),\\
& \sup_{t\in[0,T]}
\|f(t,\cdot)-f_*\|_{H^{s-\f12}} \le \delta_0,\\
&\sup_{t\in[0,T]}\Big(\|f(t)\|_{H^{s+\f12}}+\|\partial_tf(t)\|_{H^{s-\f12}}
+\|(\vom_*^\pm, \vj_*^\pm)(t)\|_{H^{s-1}(\Omega_*^\pm)}+|\beta^\pm_{i}(t)|+|\gamma^\pm_{i}(t)|\Big)\le M_1,\\
&\sup_{t\in[0,T]}\Big(\|\partial_{t}^2f\|_{H^{s-\f32}}
+\|(\partial_t\vom_*, \partial_t\vj_*)\|_{H^{s-2}(\Omega_*^\pm)}+|\partial_t\beta^\pm_{i}|+|\partial_t\gamma^\pm_{i}|\Big)\le M_2,
\end{align*}
together with the condition $\int_{\T^2}\pa_tf(t,x')dx'=0$.

\end{definition}

Given $(f,\vom_*^\pm, \vj_*^\pm,\beta^\pm_{i},\gamma^\pm_{i})\in \mathcal{X}
(T, M_1, M_2)$, our goal is to construct an iteration map $(\bar{f},\bar{\vom}_*, \bar{\vj}_*,\bar{\beta}^\pm_{i},$
$\bar{\gamma}^\pm_{i})=\mathcal{F}\big(f,\vom_*^\pm, \vj_*^\pm, \beta^\pm_{i},\gamma^\pm_{i}\big)
\in\mathcal{X}(T, M_1, M_2)$ with suitably chosen constants $M_1, M_2$ and $T$.

\subsection{Recover the bulk region, velocity and magnetic field}
Recall that
\begin{align*}
\Om_f^{+}=\big\{ x \in \Om| x_3 > f (t,x')\big\}, \quad\Om_f^{-}=\big\{ x \in \Om| x_3 < f (t,x')\big\},
\end{align*}
and harmonic coordinate map $\Phi_f^\pm:\Omega_*^\pm\to\Omega^\pm_f$.


We denote
\begin{align*}
\tilde{\vom}^\pm\triangleq P_{f}^{\div} (\vom _*^\pm\circ\Phi_{f }^{-1}),\quad
\tilde{\vj}^\pm\triangleq P_{f}^{\div} (\vj _*^\pm\circ\Phi_{f }^{-1}),
\end{align*}
where $P_{f }^{\div}$ is an operator, which projects a vector field
$\Omega_{f}^\pm$ to its divergence-free part. More precisely,
$P_{f }^{\div}\vom^\pm=\vom^\pm-\nabla\phi^\pm$ with
\begin{align*}
\left\{
\begin{array}{l}
\Delta\phi^\pm=\div\vom^\pm\quad\text{in}\quad \Omega_f^\pm,\\
\partial_3\phi^\pm=0\quad\text{on}\quad \Gamma^\pm,\quad \phi^\pm=0\quad\text{on}\quad \Gamma_f.
\end{array}\right.
\end{align*}
Thus, $\div P_{f }^{\div}\vom^\pm=0$ and $\ve_3\cdot P_{f}^{\div}\vom^\pm=\om_3^\pm$ on $\Gamma^\pm$. Therefore,
$P_{f }^{\div}\vom^\pm$ satisfies conditions (C1) and (C2) on $\Omega_f^\pm$, and so does $P_{f }^{\div}\vj^\pm$.
Moreover, we have
\begin{align}
&\|(\tilde{\vom}^\pm, \tilde{\vj}^\pm)\|_{H^{s-1}(\Omega_f^\pm)}\le C(M_1),\label{eq:wh-est1}\\
&\|(\pa_t\tilde{\vom}^\pm, \pa_t\tilde{\vj}^\pm)\|_{H^{s-2}(\Omega_f^\pm)}\le C\big(M_1, M_2\big),\label{eq:wh-est2}
\end{align}
Then we can define the velocity field $\vu^\pm$ and magnetic field $\vh^\pm$ by solving the following div-curl system
\begin{equation}
\left\{\begin{split}
&\curl\vu^\pm=\tilde{\vom}^\pm,\quad\div\vu^\pm =0\quad  \text{in}\quad\Om_f^{\pm},\\
&\vu^\pm\cdot\vN_f=\partial_tf\quad\text{on}\quad\Gamma_{f}, \\
& \vu^\pm\cdot\ve_3 = 0,\quad\int_{\Gamma^\pm}u_i dx'=\beta^\pm_i (i=1,2)\quad\text{on}\quad\Gamma^{\pm},
\end{split}\right.
\end{equation}
and
\begin{equation}
\left\{\begin{split}
&\curl \vh^\pm=\tilde{\vj}^\pm,\quad\div\vh^\pm=0\quad\text{in}\quad\Om_f^{\pm},\\
&\vh^\pm\cdot\vN_f = 0\quad \text{on}\quad\Gamma_{f}, \\
& \vh^\pm\cdot\ve_3 = 0, \quad \int_{\Gamma^\pm} h_i dx'=\gamma^\pm_i (i=1,2)\quad\text{on}\quad\Gamma^{\pm}.
\end{split}\right.
\end{equation}
It follows from Proposition \ref{prop:div-curl} and (\ref{eq:wh-est1}) that
\begin{align}
\|\vu^\pm\|_{H^{s}(\Omega_f^\pm)}\le &C(M_1)\big(\|\tilde{\vom}^\pm\|_{H^{s-1}(\Omega_f^\pm)}+\|\partial_tf
\|_{H^{s-\f12}}+|\beta^\pm_1|+|\beta^\pm_2|\big)\le C(M_1),\\
\|\vh^\pm\|_{H^{s}(\Omega_f^\pm)}\le &C(M_1)\big(\|\tilde{\vj}^\pm\|_{H^{s-1}(\Omega_f^\pm)}+|\gamma^\pm_1|+|\gamma^\pm_2|\big)\le C(M_1).
\end{align}
In addition, there holds
\begin{align*}
\vu^\pm(0)=\vu_0^\pm,\quad \vh^\pm(0)=\vh_0^\pm.
\end{align*}

Using the fact that
\begin{align*}
\partial_t(\vu^\pm\cdot\vN_f)=\pa_t\vu^\pm\cdot\vN_f+\vu^\pm\cdot\partial_t\vN_f
=(\partial_t\vu^\pm+\partial_3\vu^\pm\partial_tf)\cdot\vN_f+\vu^\pm\cdot\partial_t\vN_f
\end{align*}
on $\Gamma_f$, it is easy to see that $\partial_t\vu^\pm$ satisfies
\begin{equation}\nonumber
\left\{\begin{split}
&\curl\partial_t\vu^\pm=\partial_t\tilde{\vom}^\pm,\quad\div\partial_t\vu^\pm=0\quad\text{in}\quad\Om_f^{\pm},\\
&\partial_t\vu^\pm\cdot\vN_f=\partial_{t}^2f-\partial_tf\partial_3\vu^\pm
\cdot\vN_f+u_1^\pm\partial_1\partial_tf+u_2^\pm\partial_2\partial_tf
\quad\text{on}\quad\Gamma_{f}, \\
&\partial_t\vu^\pm\cdot\ve_3=0, \quad \int_{\Gamma^\pm}\partial_tu_i^\pm dx=\partial_t\beta^\pm_i(i=1,2)\quad\text{on}\quad\Gamma^{\pm}.
\end{split}\right.
\end{equation}
By Proposition \ref{prop:div-curl} again and (\ref{eq:wh-est2}), we get
\begin{align*}
\|\partial_t\vu^\pm\|_{H^{s-1}(\Omega_f^\pm)}\le {C}(M_1, M_2),
\end{align*}
which implies
\begin{align}\nonumber
\!\!\!\|\vu^\pm(t)\|_{L^\infty(\Gamma_f)}\le& \|\vu^\pm_0\|_{L^\infty(\Gamma_{f_0})}+\int_0^t\|\partial_t\vu^\pm\|_{L^\infty(\Gamma_f)}dt\\
\le& \frac{M_0}{2}+T{C}(M_1,M_2).\nonumber
\end{align}
Similarly, we can obtain
\begin{align*}
&\|\partial_t\vh^\pm(t)\|_{H^{s-1}(\Omega_f^\pm)}\le {C}(M_1, M_2),\\
&\|\vh^\pm(t)\|_{L^\infty(\Gamma_f)}\le \frac{M_0}2+T{C}(M_1,M_2).\label{h-infty}
\end{align*}
Also, we have
\begin{align*}
&\|f(t)-f_0\|_{L^\infty}\le \|f(t)-f_0\|_{H^{s-\f12}}\le T\|\partial_tf\|_{H^{s-\f12}}\le TM_1, \\
&|\Lambda(\vh^\pm,[\vu])-\Lambda(\vh_0^\pm,[\vu_0])|\le TC\big(\|\partial_t\vu^\pm\|_{L^\infty(\Gamma_f)},
\|\partial_t\vh^\pm\|_{L^\infty(\Gamma_f)}\big)\le TC(M_1,M_2).
\end{align*}

Now, we choose $T$ small enough so that
\begin{align*}
TM_1\le \min\{\delta_0,c_0\},\quad TC(M_1)+T{C}(M_1,M_2)\le \frac {M_0}2,\quad TC(M_1, M_2)\le c_0,
\end{align*}
and take $L_0=M_0$, $L_1=M_1$, $L_2={C}(M_1, M_2)$. We conclude that for any $t\in [0,T]$,

\begin{itemize}
\item $-(1-c_0)\le f(t,x')\le (1-c_0)$
\item $\Lambda(\vh^\pm,[\vu])(t)\ge c_0$;
\item $\|(\vu^\pm, \vh^\pm)(t)\|_{L^{\infty}(\Gamma_f)}\le L_0$;
\item $\|f(t)-f_*\|_{H^{s-\f12}}\le \delta_0$;
\item $\|f(t)\|_{H^{s+\f12}}+\|\pa_tf(t)\|_{H^{s-\f12}}+\|\vu^\pm(t)\|_{H^{s}(\Omega_f^\pm)}
+\|\vh^\pm(t)\|_{H^{s}(\Omega_f^\pm)}\le L_1$;
\item $\|(\partial_t\vu^\pm, \partial_t\vh^\pm)(t)\|_{L^{\infty}(\Gamma_f)}\le L_2$.
\end{itemize}
%

\subsection{Define the iteration map}

Given $(f,\vu^\pm,\vh^\pm)$ as above, let us define the iteration map.
First of all, we solve $\bar f_1$ by the linearized system (\ref{sys:linear-H}) and $(\bar\vom^\pm, \bar\vj^\pm)$ by (\ref{eq:vorticity-w-L}) and
(\ref{eq:vorticity-h-L}) with the initial data
\begin{align*}
&\left(\bar f_1(0),\bar\theta(0), \bar\vom^\pm(0) , \bar\vj^\pm(0)\right)
=\big(f_0,(\partial_tf)_I, \vom_{*I}^\pm, \vj_{*I}^\pm\big).
\end{align*}

We define
\begin{align*}
&\bar\vom_{*}^\pm=\bar\vom^\pm\circ \Phi_{f}^\pm,\quad \bar\vj_{*}^\pm=\bar\vj^\pm\circ\Phi_{f}^\pm,\\
&\bar\beta^\pm_i(t)=\beta^\pm_i(0)-\int_0^t\int_{\Gamma^\pm}\big(u_j^\pm\partial_ju_i^\pm-h_j^\pm\partial_jh_i^\pm\big)dx'd\tau,\\
&\bar\gamma^\pm_i(t)=\gamma^\pm_i(0)-\int_0^t\int_{\Gamma^\pm}\big(u_j^\pm\partial_jh_i^\pm-h_j^\pm\partial_ju_i^\pm\big)dx'd\tau.
\end{align*}

The iteration map $\mathcal{F}$ is defined as follows
\ben
\mathcal{F}\big(f,\vom_*^\pm, \vj_*^\pm, \beta^\pm_{i},\gamma^\pm_{i}\big)
\eqdefa \big(\bar{f},\bar{\vom}_*^\pm, \bar{\vj}_*^\pm,\bar{\beta}^\pm_{i},
\bar{\gamma}^\pm_{i}\big),
\een
where $\bar f$ is given by
\ben
\bar f(t,x')=\bar f_1(t,x')-\langle \bar f_1\rangle+\langle f_0\rangle.
\een
Hence, $\langle \bar f\rangle=\langle f_0\rangle$ and $\int_{\T^2}\pa_t \bar f(t,x')dx'=0$ for $t\in [0,T]$.

\begin{proposition}\label{prop:iteration map}
There exist $M_1, M_2, T>0$ depending on $c_0, \delta_0, M_0$ so that $\mathcal{F}$ is a map from $\mathcal{X}(T, M_1,M_2)$ to itself.
\end{proposition}
\begin{proof}
We check the conditions in Definition \ref{def:X}.
The initial conditions is automatically satisfied. Proposition \ref{prop:f-L} and Proposition \ref{prop:vorticity}
ensure that
\begin{align*}\nonumber
&\sup_{t\in[0,T]}\left( \|\bar{f}(t)\|_{H^{s+\f12}}+\|\partial_t\bar{f}(t)\|_{H^{s-\f12}}+\|\bar\vom_*^\pm(t)\|_{H^{s-1}
(\Omega_*^\pm)}+\|\bar\vj_*(t)\|_{H^{s-1}(\Omega_*^\pm)}\right)\\
&\le C(c_0,M_0)e^{C(M_1,M_2)T}.
\end{align*}
From the equation (\ref{sys:linear-H}), (\ref{eq:vorticity-w-L}) and (\ref{eq:vorticity-h-L}), we deduce that
\beno
\sup_{t\in[0,T]}\Big(\|\partial_{t}^2\bar f\|_{H^{s-\f32}}
+\|(\partial_t\vom_*, \partial_t\vj_*)\|_{H^{s-2}(\Omega_*^\pm)}\Big)\le C(M_1).
\eeno
Obviously, we have
\beno
&&|\bar\beta^\pm_i(t)|+|\bar\gamma^\pm_i(t)|\le M_0+TC(M_1),\\
&&|\partial_t\bar\beta^\pm_i(t)|+|\partial_t\bar\gamma^\pm_i(t)|\le C(M_1),\\
&&\|\bar f(t)-f_*\|_{H^{s-\f12}}\le \int_0^t\|\pa_t \bar f(\tau)\|_{H^{s-\f12}}d\tau.
\eeno

We first take $M_2=C(M_1)$, and then take $M_1$ large enough so that
\begin{align}
C(c_0,M_0)<M_1/2.
\end{align}
Next, we take $T$ sufficiently small depending only on $c_0, \delta_0, M_0$ so that all other conditions in Definition \ref{def:X} are satisfied.
\end{proof}

\section{Contraction of the iteration map}

\subsection{Contraction}

Let $\big(f^A, \vom_*^{\pm A}, \vj_*^{\pm A},\beta^{\pm A}_{i}, \gamma^{\pm A}_{i}\big)$ and
$\big(f^B$, $\vom_*^{\pm B}, \vj_*^{\pm B}, \beta^{\pm B}_{i}, \gamma^{\pm B}_{i}\big)$
be two elements in $\mathcal{X}(T, M_1,M_2)$, and
$\big(\bar f^C,\bar \vom_*^{\pm C}, \bar\vj_*^{\pm C}, \bar\beta^{\pm C}_{i},\bar\gamma^{\pm C}_{i}\big)=\mathcal{F}
\big(f^C$, $\vom_*^{\pm C}, \vj_*^{\pm C},\beta^{\pm C}_{i}, \gamma^{\pm C}_{i}\big)$ for $C=A,B$.

We denote by $g^D$ the difference $g^A-g^B$. For example,
$f^D={f}^A-{f}^B, \bar\vom_*^{\pm D}=\bar\vom_*^{\pm A}-\bar\vom_*^{\pm B}$.

\begin{proposition}\label{prop:contraction}
There exists $T>0$ depending on $c_0, \delta_0, M_0$ so that
\begin{align}\nonumber
&\bar E^D\triangleq\sup_{t\in[0,T]}\Big(\|\bar{f}^D(t)\|_{H^{s-\f12}}+\|\partial_t\bar{f}^D(t)\|_{H^{s-\f32}}+\|\bar\vom_*^{\pm D}(t)\|_{H^{s-2}(\Omega_*^\pm)}
\\&\qquad\qquad+\|\vj_*^{\pm D}(t)\|_{H^{s-2}(\Omega_*^\pm)}+|\bar\beta^{\pm D}_{i}(t)|+|\bar\gamma^{\pm D}_{i}(t)|\Big)\nonumber\\\nonumber
&\le\frac12\sup_{t\in[0,T]}\Big( \|{f}^D(t)\|_{H^{s-\f12}}+\|\partial_t{f}^D(t)\|_{H^{s-\f32}}
+\|\vom_*^{\pm D}(t)\|_{H^{s-2}(\Omega_*^\pm)}\\&\qquad\qquad+\|\vj_*^{\pm D}(t)\|_{H^{s-2}(\Omega_*^\pm)}
+|\beta^{\pm D}_{i}(t)|+|\gamma^{\pm D}_{i}(t)|\Big)\triangleq E^D.\nonumber
\end{align}
\end{proposition}

\begin{proof}
By the elliptic estimate, we have
\begin{align*}
\|\Phi_{f^A}^\pm-\Phi_{f^B}^\pm\|_{H^{s-1}(\Om_*^\pm)}\le C(M_1)\|f^A-f^B\|_{H^{s-\f12}}\le CE^D.
\end{align*}
Since $\vu^A$ and $\vu^B$ are defined on different region, we can not estimate their difference directly.
Thus, we introduce for $C=A,B$
$$
\vu^{\pm C}_*=\vu^{\pm C}\circ\Phi_{f^C}^{\pm},\quad
\vh^{\pm C}_*=\vh^{\pm C}\circ\Phi_{f^C}^{\pm}.
$$

Let us first show that
\begin{align}\label{eq:uh-d}
\|\vu^{\pm D}_*\|_{H^{s-1}(\Om^\pm_*)}+\|\vh^{\pm D}_*\|_{H^{s-1}(\Om^\pm_*)}\le CE^D.
\end{align}
For a vector field $\vv_*^\pm$ defined on $\Omega_*^\pm$, we define
\begin{align*}
\curl_C \vv_*^\pm=\big(\curl (\vv_*^\pm\circ(\Phi^{\pm}_{f^C})^{-1})\big) \circ\Phi_{f^C}^\pm,\quad
\div_C \vv_*^\pm =\big(\div(\vv_*^\pm\circ(\Phi^{\pm}_{f^C})^{-1}\big) \circ\Phi_{f^C}^\pm,
\end{align*}
for $C=A,B$. Then we find that for $C=A,B$,
\begin{equation}\nonumber
\left\{\begin{split}
&\curl_C \vu^{\pm C}_*=\tilde{\vom}^{\pm C}_*\quad\text{in}\quad\Om^{\pm}_*,\\
&\div_C \vu_*^{\pm C}=0\quad\text{in}\quad\Om^{\pm}_*,\\
&\vu^{\pm C}_*\cdot\vN_{f^C}=\partial_tf^C\quad\text{on}\quad\Gamma_{*},\\
&\vu^{\pm C}\cdot\ve_3 = 0,\quad \int_{\Gamma^\pm} u_i^{\pm C}dx'=\beta_i^{\pm C}\quad\text{on}\quad\Gamma^{\pm}.
\end{split}\right.
\end{equation}
Thus, we have
\begin{equation}\nonumber
\left\{\begin{split}
&\curl_A\vu^{\pm D}_*=\tilde{\vom}^{\pm D}_*+(\curl_B-\curl_A)\vu^{\pm B}_*\quad\text{in}\quad \Om^{\pm}_*,\\
&\div_A\vu^{\pm D}_*=(\div_B-\div_A)\vu^{\pm B}_*\quad\text{in}\quad\Om^{\pm}_*,\\
&\vu^{\pm D}_*\cdot\vN_{f^A}=\partial_tf^D+\vu^{\pm B}_*\cdot(\vN_{f^B}-\vN_{f^A})\quad\text{on}\quad\Gamma_{*}, \\
&\vu^{\pm D}_*\cdot\ve_3=0,\quad \int_{\Gamma^\pm}u_i^{\pm D}dx'=\beta_i^{\pm D}\quad \text{on}\quad\Gamma^{\pm}.
\end{split}\right.
\end{equation}
It is easy to show that
\begin{align*}
\|(\curl_B-\curl_A)\vu^{\pm B}_*\|_{H^{s-2}(\Omega_*^\pm)}\le&~ C\|\Phi_{f^A}^\pm-\Phi_{f^B}^\pm\|_{H^{s-1}(\Omega_*^\pm)}\\
\le& C\|f^D\|_{H^{s-\f12}}\le CE^D.
\end{align*}
Similarly, we have
\begin{align*}
&\|(\div_B-\div_A)\vu^{\pm B}_*\|_{H^{s-2}(\Omega_*^\pm)}\le CE^D,\\
&\|\vu^{\pm B}_*\cdot(\vN_{f^B}-\vN_{f^B})\|_{H^{s-\f32}}\le CE^D.
\end{align*}
Then we infer from Proposition \ref{prop:div-curl} that
\begin{align}\nonumber
\|\vu^{\pm D}_*\|_{H^{s-1}(\Om_*^\pm)}\le C\left(\|\tilde{\vom}^{\pm D}_*\|_{H^{s-2}(\Omega_*^\pm)}
+\|\partial_tf^D\|_{H^{s-1}}+E^D\right)\le CE^D.
\end{align}
Similarly, we have
\begin{align}\nonumber
\|\vh^{\pm D}_*\|_{H^{s-1}}\le& CE^D.
\end{align}

We have
\begin{align*}
  \partial_t\bar{f}_1^D=&~\bar\theta^D,\\\nonumber
\partial_t\bar{\theta}^D=&-\left((\underline{u}^{+A}_1+\underline{u}^{-A}_1)\partial_1\bar{\theta}^D
+(\underline{u}^{+A}_2+\underline{u}^{-A}_2)\partial_2\bar\theta^D\right)\\
&-\frac12\sum_{i,j=1,2}(\underline{u}_i^{+A}\underline{u}^{+A}_j-\underline{h}^{+A}_i\underline{h}^{+A}_j+\underline{u}^{-A}_i\underline{u}^{-A}_j
-\underline{h}^{-A}_i\underline{h}^{-A}_j)
\partial_i\partial_j\bar{f}_1^D+\mathfrak{R},
\label{linear:cauchy:eq-2}
\end{align*}
where
\begin{align}\nonumber
\mathfrak{R}= &-\left((\underline{u}^{+D}_1+\underline{u}^{-D}_1)\partial_1\bar{\theta}^B
+(\underline{u}^{+D}_2+\underline{u}^{-D}_2)\partial_2\bar\theta^B\right)\\\nonumber
&+\frac12\sum_{i,j=1,2}\Big((\underline{u}_i^{+A}\underline{u}^{+A}_j-\underline{h}^{+A}_i\underline{h}^{+A}_j+\underline{u}^{-A}_i\underline{u}^{-A}_j
-\underline{h}^{-A}_i\underline{h}^{-A}_j)\\\nonumber
&\qquad\qquad-(\underline{u}_i^{+B}\underline{u}^{+B}_j-\underline{h}^{+B}_i\underline{h}^{+B}_j+\underline{u}^{-B}_i\underline{u}^{-B}_j
-\underline{h}^{-B}_i\underline{h}^{-B}_j)\Big)\partial_i\partial_j\bar{f}_1^B\\
&+\mathfrak{g}^A-\mathfrak{g}^B,\nonumber
\end{align}
and for $C=A,B, $
\begin{align*}
\mathfrak{g}^C=&~\frac12(\mathcal{N}_{f^C}^+-\mathcal{N}_{f^C}^-)\widetilde{\mathcal{N}}_{f^C}^{-1}
\sum_{i,j=1,2}(\underline{u}_i^{+C}\underline{u}^{+C}_j-\underline{h}^{+C}_i\underline{h}^{+C}_j-\underline{u}^{-C}_i\underline{u}^{-C}_j
+\underline{h}^{-C}_i\underline{h}^{-C}_j)\partial_i\partial_j{f^C}\\
&~+(\mathcal{N}_{f^C}^+-\mathcal{N}_{f^C}^-)\widetilde{\mathcal{N}}_{f^C}^{-1}
\left((\underline{u}^{+C}_1-\underline{u}^{-C}_1)\partial_1\theta^C
+(\underline{u}^{+C}_2-\underline{u}^{-C}_2)\partial_2\theta^C\right)\\
&~-\mathcal{N}_{f^C}^+\widetilde{\mathcal{N}}_{f^C}^{-1}\left(\vN_{f^C}\cdot
\underline{\nabla(p_{\vu^{-C},\vu^{-C}}-p_{\vh^{-C},\vh^{-C}})}\right)
\\&~-\mathcal{N}_{f^C}^-\widetilde{\mathcal{N}}_{f^C}^{-1}\left(\vN_{f^C}
\cdot\underline{\nabla(p_{\vu^{+C}, \vu^{+C}}-p_{\vh^{+C}, \vh^{+C}})}\right).
\end{align*}
Here $\underline{v}^C(x_1,x_2)$ is interpreted as $v(x_1,x_2, f^C(x_1,x_2))$ for $v=u_i^\pm, h_i^\pm$.

It is easy to verify that
\beno
\|\mathfrak{R}\|_{H^{s-\f32}}\le CE^D.
\eeno
We denote
\begin{align}
\bar F_{s}^D(\partial_t\bar{f}_1^D,\bar{f}_1^D)\triangleq&\big\|(\partial_t+w^A_i\partial_i)\langle \na\rangle^{s-\f32}\bar{f}_1^D\big\|_{L^2}^2
-\big\|v_i^A\partial_i\langle \na\rangle^{s-\f32}\bar{f}_1^D\big\|_{L^2}^2\\
&+\frac12\big\|h^{+A}_i\partial_i\langle \na\rangle^{s-\f32}\bar{f}_1^D\big\|_{L^2}^2
+\frac12\big\|h^{-A}_i\partial_i\langle \na\rangle^{s-\f32}\bar{f}_1^D\big\|_{L^2}^2.\nonumber
\end{align}
Then a similar proof of Proposition \ref{prop:f-L} gives
\begin{align*}
\frac{d}{dt}\Big(\bar F_{s}^D(\partial_t\bar{f}_1^D,\bar{f}_1^D)+\|\bar{f}_1^D\|_{L^2}^2+\|\partial_t\bar{f}_1^D\|_{L^2}^2\Big)
\le  C\big(E^D+\bar E_1^D).
\end{align*}
where
\begin{align*}
\bar E_1^D&~=\sup_{t\in[0,T]}\Big(\|\bar{f}_1^D(t)\|_{H^{s-\f12}}+\|\partial_t\bar{f}_1^D(t)\|_{H^{s-\f32}}\Big).
\end{align*}
Recall that
\begin{align*}
\|\bar{f}_1^D\|_{H^{s-\f12}}^2+\|\partial_t\bar{f}_1^D\|_{H^{s-\f32}}^2\le C
\Big(\bar E_{s}^D(\bar{f}_1^D,\partial_t\bar{f}_1^D)+\|\bar{f}_1^D\|_{L^2}^2
+\|\partial_t\bar{f}_1^D\|_{L^2}^2\Big).
\end{align*}
Thus, we obtain
\begin{align*}
\sup_{t\in[0,T]}\left(\|\bar{f}_1^D(t)\|_{H^{s-1}}+\|\partial_t\bar{f}_1^D
(t)\|_{H^{s-\f32}}\right)\le C(e^{CT}-1)E^D,
\end{align*}
which implies
\begin{align}\label{eq:f-d}
\sup_{t\in[0,T]}\left(\|\bar{f}^D(t)\|_{H^{s-1}}+\|\partial_t\bar{f}^D
(t)\|_{H^{s-\f32}}\right)\le C(e^{CT}-1)E^D.
\end{align}

Similar to the proof of Proposition \ref{prop:vorticity}, we can show that
\begin{align}\label{eq:wj-d}
\sup_{t\in[0,T]}\left(\|\bar\vom_*^D(t)\|_{H^{s-2}(\Omega_*^\pm)}+\|\bar\vj_*^D\|_{H^{s-2}(\Omega_*^\pm)}\right)\le C(e^{CT}-1)E^D.
\end{align}

Using the equation
\begin{align*}
\bar\beta^{\pm C}_i(t)=\beta^{\pm C}_i(0)+\int_0^t\int_{\Gamma^\pm}u_j^{\pm C}\big(\partial_ju_i^{\pm C}-h^{\pm C}_j\partial_jh^{\pm C}_i\big)dx'd\tau,
\end{align*}
we have
\begin{align}\label{eq:be-d}
|\bar\beta^{\pm D}_i(t)|\le|\beta^{\pm D}_{iI}|+TCE^D.
\end{align}
Similarly, we have
\begin{align}\label{eq:ga-d}
|\bar\gamma^{\pm D}_i(t)|\le|\gamma^{\pm D}_{iI}|+TCE^D.
\end{align}

Thus, we deduce from (\ref{eq:uh-d}) and (\ref{eq:f-d})--(\ref{eq:ga-d}) that
\begin{align*}
\bar{E}^D&\le C(e^{CT}-1+T)E^D.
\end{align*}
Then the proposition follows by taking $T$ small enough depending on $c_0, \delta_0, M_0$.
\end{proof}

\subsection{The limit system}

Proposition \ref{prop:iteration map} and Proposition \ref{prop:contraction} ensure that the map $\mathcal{F}$ has a unique fixed point
$(f,\vom^\pm, \vj^\pm,\beta^\pm_i, \gamma^\pm_i)$
in $\mathcal{X}(T,M_1,M_2)$. From the construction of $\mathcal{F}$, we know that $(f,\vom^\pm, \vj^\pm,\beta^\pm_i, \gamma^\pm_i)$
satisfies
\begin{align}
\partial_tf&~=\theta-\langle\theta\rangle,\label{eq:limit-theta}\\\nonumber
\partial_t\theta&~=
-\big((\underline{u}^+_1+\underline{u}^-_1)\partial_1\theta+(\underline{u}^+_2+\underline{u}^-_2)\partial_2\theta\big)
\\\nonumber
&\quad+\frac12\sum_{i,j=1,2}\big(\underline{u}_i^+\underline{u}^+_j-\underline{h}^+_i\underline{h}^+_j+\underline{u}^-_i\underline{u}^-_j-\underline{h}^-_i\underline{h}^-_j\big)
\partial_i\partial_jf\\\nonumber
&\quad+\frac12(\mathcal{N}^+_f-\mathcal{N}^-_f)\widetilde{\mathcal{N}}^{-1}_f
\sum_{i,j=1,2}\big(\underline{u}_i^+\underline{u}^+_j-\underline{h}^+_i\underline{h}^+_j-\underline{u}^-_i\underline{u}^-_j
+\underline{h}^-_i\underline{h}^-_j\big)\partial_i\partial_jf
\\\nonumber
&\quad+(\mathcal{N}^+_f-\mathcal{N}^-_f)\widetilde{\mathcal{N}}^{-1}_f
\big((\underline{u}^+_1-\underline{u}^-_1)\partial_1\theta+(\underline{u}^+_2-\underline{u}^-_2)\partial_2\theta\big)\\\nonumber
&\quad-\mathcal{N}^+_f\widetilde{\mathcal{N}}^{-1}_f(\vN_f\cdot\underline{\nabla(p_{\vum,\vum}-p_{\vhm,\vhm})})\\
&\quad-\mathcal{N}^-_f\widetilde{\mathcal{N}}^{-1}_f(\vN_f\cdot\underline{\nabla(p_{\vup, \vup}-p_{\vhp, \vhp})}),
\end{align}
where $(\vu^\pm,\vh^\pm)$ sovles the div-curl system
\ben
\left\{
\begin{array}{l}
\curl \vu^\pm=P_f^{div}\vom^\pm,\quad \div\vu^\pm=0\quad \text{in} \quad \Om_f^\pm,\\
\vu^\pm\cdot\vN_f=\partial_tf\quad\text{on}\quad\Gamma_f,\\
u_3^\pm=0\quad\text{on}\quad\Gamma^\pm,\quad \int_{\Gamma^\pm}u_i^\pm dx'=\beta_i^\pm,\\
\partial_t\beta_i^\pm~=-\int_{\Gamma^\pm}(u_j^\pm\partial_ju_i^\pm-h_j^\pm\partial_jh_i^\pm)dx',
\end{array}\right.
\een
and
\ben
\left\{
\begin{array}{l}
\curl\vh^\pm=P_f^{div}\vj^\pm,\quad \div \vh^\pm=0\quad \text{in} \quad \Om_f^\pm,\\
\vh^\pm\cdot\vN_f=0\quad \text{ on}\quad\Gamma_f,\\
h_3^\pm=0\quad\text{on}\quad\Gamma^\pm,\quad \int_{\Gamma^\pm}h_i^\pm dx'=\gamma_i^\pm,\\
\partial_t\gamma_i^\pm=-\int_{\Gamma^\pm}(u_j^\pm\partial_jh_i^\pm-h_j\partial_ju_i^\pm)dx',
\end{array}\right.
\een
and the vorticity $(\vom^\pm, \vj^\pm)$ satisfies
\begin{align}
&\partial_t\vom^\pm+\vu^\pm\cdot\nabla\vom^\pm=\vh^\pm\cdot\nabla\vj^\pm+\vom^\pm\cdot\nabla\vu^\pm-\vj^\pm\cdot\nabla\vh^\pm,\\
&\partial_t\vj^\pm+\vu^\pm\cdot\nabla\vj^\pm=\vh^\pm\cdot\nabla\vom^\pm+\vj^\pm\cdot\nabla\vu^\pm
-\vom^\pm\cdot\nabla\vh^\pm-2\nabla u_i^\pm\times\nabla h_i^\pm.\label{eq:limit-h}
\end{align}

Recall that $p_{\vupm_1, \vupm_2}$ is determined by the elliptic equation
\ben\label{eq:limit-pressure}
\left\{
\begin{array}{l}
\Delta p_{\vupm_1, \vupm_2}= -\mathrm{tr}(\nabla\vu_1^\pm\nabla\vu_2^\pm)
\quad \text{in}\quad\Omega^\pm_f,\\
p_{\vupm_1, \vupm_2}=0\quad\text{on}\quad\Gamma_f,
\quad\ve_3\cdot\nabla p_{\vupm_1, \vupm_2}=0\quad\text{on}\quad\Gamma^\pm.
\end{array}\right.
\een

\section{From the limit system to the current-vortex sheet system}

It is highly nontrivial whether the limit system (\ref{eq:limit-theta})-(\ref{eq:limit-h}) is equivalent to the current-vortex sheet system (\ref{cvs})-(\ref{cvs0}).
The proof is split into several steps.

\medskip

{\bf Step 1.} $\curl\vu^\pm=\vom^\pm$ and $\curl\vh^\pm=\vj^\pm$\medskip

From the fact that $\div\vu^\pm=\div\vh^\pm=0$, it is easy to get that
\begin{align*}
&\partial_t\div\vom^\pm+\vu^\pm\cdot\nabla\div\vom^\pm=\vh^\pm\cdot\nabla\div\vj^\pm,\\
&\partial_t\div\vj^\pm+\vu^\pm\cdot\nabla\div\vj^\pm=\vh^\pm\cdot\nabla\div\vom^\pm,
\end{align*}
which imply that $\div\vom^\pm=\div\vj^\pm=0$,  and thus $\curl\vu^\pm=\vom^\pm, \curl\vh^\pm=\vj^\pm$.
\medskip

{\bf Step 2.} Determination of the pressure\medskip

We define the projection operator $\mathcal{P}:L^2(\BT)\to L^2(\BT)$
\begin{align*}
\mathcal{P}g = g-\langle g\rangle.
\end{align*}

We introduce the pressure $p^\pm$ of the fluid by
\begin{align}
p^\pm=\mathcal{H}^\pm\underline{p}^\pm+p_{\vupm, \vupm}-p_{\vhpm, \vhpm},
\end{align}
where
\begin{align*}
\underline{p}^+=\underline{p}^-=\widetilde{\mathcal{N}}^{-1}_f\mathcal{P}(g^--g^+)
\end{align*}
with
\begin{align*}
g^\pm=&~2(\underline{u}^\pm_1\partial_1\theta+\underline{u}^\pm_2\partial_2\theta)+\vN\cdot\underline{\nabla(p_{\vupm, \vupm}+p_{\vhpm, \vhpm})}\\
&+\sum_{i,j=1,2}\big(\underline{u}_i^\pm \underline{u}^\pm_j-\underline{h}^\pm_i\underline{h}^\pm_j\big)\partial_i\partial_jf.
\end{align*}
Here we should be careful that $g^\pm$ may not have zero mean.

\medskip

{\bf Step 3.} The velocity equation
\medskip

We denote
\begin{align*}
\vw^\pm\triangleq\partial_t\vu^\pm+\vu^\pm\cdot\nabla\vu^\pm-\vh^\pm\cdot\nabla\vh^\pm+\nabla p^\pm.
\end{align*}
We will show that
\ben\label{eq:limit-u}
\left\{
\begin{array}{l}
\div\vw^\pm=0,\quad \curl\vw^\pm=0\quad\text{in}\quad\Omega^\pm_f,\\
\vw^\pm\cdot\vN_f=0\quad\text{on}\quad\Gamma_f,\\
w_3^\pm=0\quad\text{on}\quad \Gamma^\pm,\quad\int_{\Gamma^\pm}w_i^\pm dx'=0(i=1,2).
\end{array}\right.
\een
This in particular implies that
\begin{align*}
\partial_t\vu^\pm+\vu^\pm\cdot\nabla\vu^\pm-\vh^\pm\cdot\nabla\vh^\pm+\nabla p^\pm=0\quad \text{in}\quad\Omega^\pm_f.
\end{align*}

First, by the definition of $p^+$, we have
\begin{align}\label{consis:div}
&\div \partial_t\vu^+= 0 =\div (-\vu^+\cdot\nabla\vu^++\vh^+\cdot\nabla\vh^++\nabla p^+).
\end{align}
A direct computation yields that
\begin{align*}\nonumber
\curl \partial_t\vu^+&=\partial_t\curl\vu^+=\partial_t\vom^+\\\nonumber
&=-\vu^+\cdot\nabla\vom^++\vh^+\cdot\nabla\vj^++\vom^+\cdot\nabla\vu^+-\vj^+\cdot\nabla\vh^+\\
&=\curl (-\vu^+\cdot\nabla\vu^++\vh^+\cdot\nabla\vh^++\nabla p^+).
\end{align*}
Thus, we obtain
\beno
\div\vw^+=0,\quad \curl\vw^+=0\quad\text{in}\quad\Omega^+_f.
\eeno

As $u_3=0, h_3=0$ on $\Gamma^+$, we have
\begin{align}\label{eq:limit-wb}
w_3^\pm=\partial_tu_3^\pm+u_i^\pm\partial_iu_3^\pm-h_i\partial_ih_3^\pm-\partial_3p^\pm=0\quad\text{on}\quad\Gamma^+.
\end{align}
Moreover, it holds that for $i=1,2$,
\begin{align*}
\int_{\Gamma^+}w_i^\pm dx'=\int_{\Gamma^+}\big(\partial_tu_i^\pm+u_j^\pm\partial_ju_i^\pm-h_j^\pm\partial_jh_i^\pm\big)dx'=0.
\end{align*}

By the converse computations in Section 5.1, we get
\begin{align}\nonumber
&\mathcal{P}\Big\{-2(u^+_1\partial_1\theta+u^+_2\partial_2\theta)-\vN\cdot\nabla p^+
-\sum_{i,j=1,2}u_i^+u^+_j\partial_i'\partial_j'f+\sum_{i,j=1,2}h^+_ih^+_j\partial_i\partial_jf\Big\}\\\nonumber
&=-2\mathcal{P}(u^+_1\partial_1\theta+u^+_2\partial_2\theta)+\mathcal{N}^+_f\Big(\widetilde{\mathcal{N}}^{-1}_f(g^--g^+)\Big)\\\nonumber
&\quad -\mathcal{P}\sum_{i,j=1,2}u_i^+u^+_j\partial_i\partial_jf+
\mathcal{P}\sum_{i,j=1,2}h^+_ih^+_j\partial_i\partial_jf\\\nonumber
&=-2\mathcal{P}(u^+_1\partial_1\theta+u^+_2\partial_2\theta)-\mathcal{N}^+_f\widetilde{\mathcal{N}}^{-1}_f
\Big(2(u^-_1\partial_1\theta+u^-_2\partial_2\theta)\\\nonumber
&\quad-\frac12\sum_{i,j=1,2}(u_i^+u^+_j-h^+_ih^+_j+u^-_iu^-_j-h^-_ih^-_j)
\partial_i\partial_jf\\\nonumber
&\quad+\frac12(\mathcal{N}^+_f-\mathcal{N}^-_f)\widetilde{\mathcal{N}}^{-1}_f
\sum_{i,j=1,2}(u_i^+u^+_j-h^+_ih^+_j-u^-_iu^-_j+h^-_ih^-_j)\partial_i\partial_jf
\\\nonumber
&\quad+\vN_f\cdot\nabla(p_{\vum, \vum}+p_{\vhm, \vhm})-2(u^+_1\partial_1\theta+u^+_2\partial_2\theta)
-\vN_f\cdot\nabla(p_{\vup, \vup}+p_{\vhp, \vhp})\Big)\\\nonumber
&=-\mathcal{P}\big((u^+_1+u^-_1)
\partial_1\theta+(u^+_2+u^-_2)\partial_2\theta\big)\\\nonumber
&\quad+\frac12\mathcal{P}\sum_{i,j=1,2}(u_i^+u^+_j-h^+_ih^+_j+u^-_iu^-_j-h^-_ih^-_j)\partial_i\partial_jf\\\nonumber
&\quad+\frac12(\mathcal{N}^+_f-\mathcal{N}^-_f)\widetilde{\mathcal{N}}^{-1}_f
\sum_{i,j=1,2}(u_i^+u^+_j-h^+_ih^+_j-u^-_iu^-_j+h^-_ih^-_j)\partial_i\partial_jf
\\\nonumber
&\quad+(\mathcal{N}^+_f-\mathcal{N}^-_f)\widetilde{\mathcal{N}}^{-1}_f
((u^+_1-u^-_1)\partial_1\theta+(u^+_2-u^-_2)\partial_2\theta)\\\nonumber
&\quad-\mathcal{N}^+_f\widetilde{\mathcal{N}}^{-1}_f(\vN_f\cdot\nabla(p_{\vum,\vum}-p_{\vhm,\vhm}))
-\mathcal{N}^-_f\widetilde{\mathcal{N}}^{-1}_f(\vN_f\cdot\nabla(p_{\vup, \vup}-p_{\vhp, \vhp}))\\
&=\mathcal{P}\partial_t\theta.\nonumber
\end{align}
Recalling that $\mathcal{P}\partial_t\theta=\mathcal{P}\partial_t^2f=\mathcal{P}\partial_t(\vu^+\cdot\vN_f)$, we obtain
\begin{align*}\nonumber
\mathcal{P}\Big\{\partial_t(&\vu^+\cdot\vN_f)+2\big(u^+_1\partial_1'
(\vu^+\cdot\vN_f)+u^+_2\partial_2'(\vu^+\cdot\vN_f)\big)+\vN_f\cdot\nabla p^+\\
&+\sum_{i,j=1,2}(h^+_ih^+_j-u_i^+u^+_j)\partial_i\partial_jf\Big\}=0,
\end{align*}
from which and the fact
\[\partial_t\vN_f=(-\partial_1\partial_tf,-\partial_2\partial_tf, 0)
=\big(-\partial_1(\vu^+\cdot\vN_f),-\partial_2(\vu^+\cdot\vN_f), 0\big),\]
we deduce that
\begin{align*}\nonumber
\mathcal{P}\Big\{(\partial_t\vu^++\partial_3\vu^+\partial_tf)\cdot\vN_f+
\big(u^+_1\partial_1'(\vu^+\cdot\vN_f)+u^+_2\partial_2'(\vu^+\cdot\vN_f)\big)\\
+\sum_{i,j=1,2}(h^+_ih^+_j-u_i^+u^+_j)\partial_i\partial_jf+\vN_f\cdot\nabla p^+\Big\}=0.
\end{align*}
Then we infer from Lemma \ref{rel-uh} that
\begin{align*}
\mathcal{P}\big\{(\partial_t\vu^+\cdot\vN_f+\vu^+\cdot\nabla\vu^+-\vh^+\cdot\nabla\vh^++\nabla p^+)|_{\Gamma_f}\cdot\vN_f\big\}=0.
\end{align*}
On the other hand, (\ref{consis:div}) and (\ref{eq:limit-wb}) imply that
\[
\int_{\BT}(\partial_t\vu^++\vu^+\cdot\nabla\vu^+-\vh^+\cdot\nabla\vh^++\nabla p^+)|_{\Gamma_f}\cdot\vN_fdx'=0,\]
which gives rise to
\begin{align*}
\vw^+\cdot\vN_f=(\partial_t\vu^++\vu^+\cdot\nabla\vu^+-\vh^+\cdot\nabla\vh^++\nabla p^+)\cdot\vN_f=0\quad\text{on}\quad \Gamma_f.
\end{align*}
This shows that $\vw^+$ satisfies the system (\ref{eq:limit-u}). The proof for $\vw^-$ is similar.

\medskip
{\bf Step 4.} The magnetic field equation
\medskip

It suffices to show that
\ben\label{eq:limit-H}
\left\{
\begin{array}{l}
\div \textbf{H}^\pm=0,\quad \curl\textbf{H}^\pm=0\quad\text{in}\quad\Omega^\pm_f,\\
\textbf{H}^\pm\cdot\vN_f=0\quad\text{on}\quad\Gamma_f,\\
H_3^\pm=0\quad\text{on}\quad \Gamma^\pm,\quad\int_{\Gamma^\pm}H_i^\pm dx'=0(i=1,2),
\end{array}\right.
\een
where $\textbf{H}^\pm=\partial_t\vh^\pm-\vh^\pm\cdot\nabla\vu^\pm+\vu^\pm\cdot\nabla\vh^\pm$. This implies that
\beno
\partial_t\vh^\pm-\vh^\pm\cdot\nabla\vu^\pm+\vu^\pm\cdot\nabla\vh^\pm=0\quad\text{in}\quad\Omega_f^\pm.
\eeno

From the fact that $\vh^+\cdot\vN_f=0$ on $\Gamma_f$, we deduce that
\begin{align}\nonumber
0&~=\partial_t(\vhp\cdot\vN_f)+u_1\partial_1'(\vhp\cdot\vN_f)+u_2\partial_2'(\vhp\cdot\vN_f)\\\nonumber
&~=(\partial_t\vhp+\partial_3\vhp\partial_tf)\cdot\vN_f+\vhp\cdot\partial_t\vN_f
+(u_1\partial_1\vhp+u_1\partial_3\vhp\partial_1f)\cdot\vN_f\\\nonumber
&\quad+u_1\vhp\cdot\partial_1\vN_f+ (u_2\partial_2\vhp+u_2\partial_3\vhp\partial_2f)\cdot\vN_f+u_2\vhp\cdot\partial_2\vN_f\\\nonumber
&~=(\partial_t\vhp+\vu^+\cdot\nabla\vh^+)\cdot\vN_f+\partial_tf\partial_3\vhp\cdot\vN_f+\vhp\cdot\partial_t\vN_f\\\nonumber
&\quad+u_1\vhp\cdot\partial_1\vN_f +u_2\vhp\cdot\partial_2\vN_f+(u_1\partial_3
\vhp\partial_1f+u_2\partial_3\vhp\partial_2f-u_3\partial_3\vhp)\cdot\vN_f\\
&~=(\partial_t\vhp+\vu^+\cdot\nabla\vh^+)\cdot\vN_f+\vhp\cdot\partial_t\vN_f
+u_1\vhp\cdot\partial_1\vN_f +u_2\vhp\cdot\partial_2\vN_f.\nonumber
\end{align}
On the other hand, we have
\begin{align*}\nonumber
&\vhp\cdot\partial_t\vN_f+u_1\vhp\cdot\partial_1\vN_f +u_2\vhp\cdot\partial_2\vN_f\\
&=-h^+_1\partial_1(\vup\cdot\vN_f)-h^+_2\partial_2(\vup\cdot\vN_f)+u_1\vhp\cdot\partial_1\vN_f +u_2\vhp\cdot\partial_2\vN_f\\
&=-h^+_1(\partial_1\vup+\partial_3\vup\partial_1f)\cdot\vN_f-
h^+_2(\partial_2\vup+\partial_3\vup\partial_2f)\cdot\vN_f\\
&\quad -h^+_1\vup\cdot\partial_1\vN_f-h^+_2\vup\cdot\partial_2\vN_f-\sum_{i,j=1,2}u_ih_j\partial_i\partial_jf\\
&=-h^+_1(\partial_1\vup+\partial_3\vup\partial_1f)\cdot\vN_f-
h^+_2(\partial_2\vup+\partial_3\vup\partial_2f)\cdot\vN_f\\
&=-(\vhp\cdot\nabla\vup)\cdot\vN_f-(\partial_3\vup\cdot\vN_f)(\vhp\cdot\vN_f)\\
&=-(\vhp\cdot\nabla\vup)\cdot\vN_f.
\end{align*}
Thus, we deduce that
\begin{align}
\big(\partial_t\vhp-\vhp\cdot\nabla\vup+\vu^+\cdot\nabla\vh^+\big)\cdot\vN_f=0\quad \text{on}\quad \Gamma_f.\nonumber
\end{align}

Moreover, we have
\begin{align}
\div(\partial_t\vhp)=0=\div(\vhp\cdot\nabla\vup-\vu^+\cdot\nabla\vh^+),\nonumber
\end{align}
and by (\ref{eq:limit-h}),
\begin{align}\nonumber
\curl(\partial_t\vhp)&=\partial_t\vj^+\\\nonumber
&=\curl(-\vu^+\cdot\nabla\vj^++\vh^+\cdot\nabla\vom^++\vj^+\cdot\nabla\vu^+
-\vom^+\cdot\nabla\vh^+-2\nabla u^+_i\times\nabla h^+_i)\\
&=\curl(\vhp\cdot\nabla\vup-\vu^+\cdot\nabla\vh^+).\nonumber
\end{align}
As $u_3=0, h_3=0$ on $\Gamma^\pm$, we have $H_3^\pm=0$ on $\Gamma^\pm$.
Moreover, it holds that for $i=1,2$,
\begin{align*}
\int_{\Gamma^+}H_i^\pm dx'=\int_{\Gamma^+}\big(\partial_th_i^\pm+u_j^\pm\partial_jh_i^\pm-h_j\partial_ju_i^\pm\big)dx'=0.
\end{align*}
This shows that  $\textbf{H}^+$ satisfies the system (\ref{eq:limit-u}). The proof for $\textbf{H}^-$ is similar.\medskip

Step 1-Step 4 ensure that $(\vu^\pm,\vh^\pm, f, p^\pm)$ is a solution of the system (\ref{cvs})-(\ref{cvs0}).

\section{Appendix}

\subsection{Paradifferential operator}
Let us recall some basic facts on paradifferential operator from \cite{Met}.

We first introduce the definition of the symbol with limited spatial smoothness.
We denote by $W^{k,\infty}(\T^d)$ the usual Sobolev spaces for $k\in \mathbf{N}$, and
the H\"{o}lder space with exponent $k$ for $k\in (0,1)$.

\begin{definition}
Given $\mu\in [0,1]$ and $m\in \mathbf{R}$, we denote by $\Gamma^m_\mu(\mathbf{T}^d)$ the space of locally bounded functions $a(x,\xi)$ on
$\mathbf{T}^d\times \mathbf{R}^d \backslash\{0\}$, which are $C^\infty$ with respect to $\xi$ for $\xi\neq 0$ and such that, for all $\alpha\in \mathbf{N}^d$ and all
$\xi\neq 0$, the function $x\rightarrow \pa_\xi^\al a(x, \xi)$ belongs to $W^{\mu,\infty}$ and there exists a constant $C_\al$ such that
\beno
\|\pa_\xi^\al a(\cdot,\xi)\|_{W^{\mu,\infty}}\leq C_\al(1+|\xi|)^{m-|\al|}\quad \textrm{for any} \quad |\xi|\geq \f12.
\eeno
The semi-norm of the symbol is defined by
\beno
M^m_\mu(a)\eqdefa\sup_{|\alpha|\leq 3d/2+1+\mu}\sup_{|\xi|\geq1/2} \|(1+|\xi|)^{|\alpha|-m}\pa^\alpha_\xi a(\cdot,\xi)\|_{W^{\mu,\infty}}.
\eeno
Especially, if $a$ is a function independent of $\xi$, then
\beno
M^m_\mu(a)=\|a\|_{W^{\mu,\infty}}.
\eeno
\end{definition}

Given a symbol $a$, the paradifferential operator $T_a$ is defined by
\ben\label{paradiff}
\widehat{T_au}(\xi)\eqdefa (2\pi)^{-d}\int \chi(\xi-\eta,\eta)\widehat{a}(\xi-\eta,\eta)\psi(\eta)\widehat{u}(\eta) d \eta,
\een
where $\widehat a(\theta,\xi)$ is the Fourier transform of $a$ with respect to the first variable;
the $\chi(\theta,\xi)\in C^\infty(\R^d\times \R^d)$ is an admissible cut-off function: there exists $\e_1,\e_2$ such that $0<\e_1<\e_2$ and
\beno
\chi(\theta,\eta)=1 \quad\textrm{for}\quad |\theta|\leq \e_1 |\eta|, \quad \chi(\theta,\eta)=0 \quad\textrm{for} \quad|\theta|\geq \e_2 |\eta|,
\eeno
and such that for any $(\theta,\eta)\in \R^d\times \R^d$,
\beno
|\pa_\theta^\al \pa_\eta^\beta\chi(\theta,\eta)|\leq C_{\al,\beta}(1+|\eta|)^{-|\al|-|\beta|}.
\eeno
The cut-off function $\psi(\eta)\in C^\infty(\R^d)$ satisfies
\beno
\psi(\eta)=0 \quad\textrm{for}\quad |\eta|\leq 1, \quad \psi(\eta)=1 \quad\textrm{for}\quad |\eta|\geq 2.
\eeno
Here we will take the admissible cut-off function $\chi(\theta,\eta)$ as follows
\beno
\chi(\theta,\eta)=\sum_{k=0}^\infty\zeta_{k-3}(\theta)\varphi_k(\eta),
\eeno
where $\zeta(\theta)=1$ for $|\theta|\le 1.1$ and $\zeta(\theta)=0$ for $|\theta|\ge 1.9$; and
\beno
&&\zeta_k(\theta)=\zeta(2^{-k}\theta)\quad \textrm{for}\quad k\in\Z,\\
&&\varphi_0=\zeta,\quad \varphi_k=\zeta_k-\zeta_{k-1}\quad \textrm{for}\quad k\ge 1.
\eeno

We also introduce the Littlewood-Paley operators $\Delta_k, S_k$ defined by
\beno
&&\Delta_k u={\mathcal F}^{-1}\big(\varphi_k(\xi)\widehat u(\xi)\big)\quad\textrm{ for }k\ge 0, \quad \Delta_ku=0\quad\textrm{ for }k<0,\\
&&S_ku=\sum_{\ell\le k}\Delta_\ell u\quad\textrm{ for }k\in\Z.
\eeno
In the case when the function $a$ depends only on the first variable $x$ in $T_au$, we  take $\psi=1$. Then $T_au$ is just the usual
Bony's paraproduct  defined by
\ben\label{paraproduct}
T_au=\sum_{k}S_{k-3}a\Delta_ku.
\een

We have the following well-known Bony's decomposition(see \cite{BCD}):
\ben\label{Bony}
au=T_au+T_ua+R(u,a),
\een
where the remainder term  $R(u,a)$ is defined by
\ben
&&R(u,a)=\sum_{|k-\ell|\le 2}\Delta_{k}a\Delta_\ell u.
\een

We have the following classical estimate for $R(u,a)$.

\begin{lemma}
\label{lem:bony}
It holds that
\begin{itemize}

\item[1.] If $s\in \R$ and $\sigma<\f d2$, then we have
\beno
\|T_au\|_{H^s}\le C\min\Big(\|a\|_{L^\infty}\|u\|_{H^s}, \|a\|_{H^{\sigma}}\|u\|_{H^{s+\f d2-\sigma}}, \|a\|_{H^\f d2}\|u\|_{H^{s+\e}}\Big)
\eeno
for any $\e>0$.

\item[2.] If $s>0$ and $s_1,s_2\in \R$ with $s_1+s_2=s+\f d2$, then we have
\beno
\|R(u,a)\|_{H^s}\le C\|a\|_{H^{s_1}}\|u\|_{H^{s_2}}.
\eeno
\end{itemize}
\end{lemma}

Finally, let us recall the symbolic calculus of paradifferential operator in Sobolev space.

\begin{proposition}\label{prop:symcal}
Let $m, m'\in \R$.

\begin{itemize}

\item[1.]If $a\in\Gamma^m_0(\mathbf{T}^d)$, then for any $s\in\R$,
\beno
\|T_a\|_{H^s\rightarrow H^{s-m}}\leq CM_0^m(a).
\eeno

\item[2.] If $a\in\Gamma^{m}_{\rho}(\mathbf{T}^d), b \in\Gamma^{m'}_{\rho}(\mathbf{T}^d)$ for $\rho>0$, then for any $s\in\R$,
\beno
\|T_aT_b-T_{a\#b}\|_{H^s\rightarrow H^{s-m-m'+\rho}}\leq CM_\rho^{m_1}(a)M_0^{m'}(b)+CM_0^{m_1}(a)M_\rho^{m'}(b),
\eeno
where $a\#b=\sum_{|\al|<\rho}\pa_\xi^\al a(x,\xi)D_x^\al b(x,\xi), D_x=\f {\pa_x} i$.

\item[3.] If $a\in\Gamma^m_\rho(\mathbf{T}^d)$ for $\rho\in (0,1]$, then for any $s\in\R$,
\beno
\|T_{a^*}-(T_a)^*\|_{H^s\rightarrow H^{s-m+\rho}}\leq CM_\rho^m(a).
\eeno
\end{itemize}
Here $(T_a)^*$ is the adjoint operator of $T_a$, and $C$ is a constant independent of $a,b$.
\end{proposition}

A direct application of Proposition \ref{prop:symcal}  gives the following commutator estimate.

\begin{lemma}\label{lem:commutator}
If $s>1+\f d2$, then we have
\beno
\big\|[a, \langle\na\rangle^s]u\big\|_{L^2}\le C\|a\|_{H^{s}}\|u\|_{H^{s-1}}.
\eeno
\end{lemma}

\subsection{Elliptic estimates in a strip}

Let $\cS_f\triangleq\big\{(x,y): x\in \T^2,-1<y<f(x)\big\}$ be a strip,
where $f(x)$ satisfies
\ben\label{ass:f-app}
1+f(x)\ge c_0>0\quad \textrm{for }x\in \T^2.
\een

We consider the elliptic boundary value problem in $\cS_f$:
\ben\label{eq:elliptic}
\left\{\begin{array}{l}
\Delta_{x,y}\Phi=0\quad \text{in}\quad\cS_f,\\
\Phi(x,f(x))=\psi(x) \quad \text{for}\quad x\in\mathbb{T}^2,\\
\partial_y\Phi(x,-1)=0\quad \text{for}\quad x\in\mathbb{T}^2.
\end{array}\right.
\een
The Lax-Milgram theorem ensures that for $\phi(x)\in H^\f12(\T^2)$, there exits a unique weak solution $\Phi(x,z)\in H^1(\cS_f)$ satisfying
\ben
\|\Phi\|_{H^1(\cS_f)}\le C\|\phi\|_{H^\f12},
\een
where the constant $C$ depends on $c_0$ and $\|f\|_{W^{1,\infty}}$.

The following elliptic estimate is classical(see also the proof of Proposition \ref{prop:elliptic}).

\begin{proposition}\label{prop:elliptic-c}
Let $\Phi\in H^1(\cS_f)$ be a weak solution of (\ref{eq:elliptic}).
Assume that $f\in H^{s+\f12}(\T^d)$ for $s>\f d 2+\f12$.
Then for any integer $\sigma\in [0, s]$, it holds that
\begin{align}\label{eq:ellip-app}
&\|\Phi\|_{H^{\sigma+1}(\cS_f)}\leq C\big(c_0, \|f\|_{H^{s+\f12}}\big)\|\psi\|_{H^{\sigma+\f12}}.
\end{align}
\end{proposition}

In order to estimate the Dirichlet-Neumann(DN)  operator, we need to establish a different form of elliptic estimate.  We will follow the method introduced by Alazard, Burq and Zuily \cite{ABZ}. The idea is to transform the elliptic estimate into the parabolic estimate by decoupling the elliptic equation  into a forward and a backward parabolic evolution equation.

We first flatten the strip $\cS_f$ by a regularized mapping
\beno
(x,z)\in S\triangleq\T^2\times I\longmapsto (x,\rho_\delta(x,z))\in\cS_f,
\eeno
where $I=[-1,0]$ and $\rho_\delta$ with $\delta>0$ is given by
\begin{eqnarray}\label{mapping}
\rho_\delta(x,z)=z+(1+z)e^{\delta z|D|}f(x).
\end{eqnarray}
It is easy to verify that there exists $\delta>0$ depending on $c_0$ and $\|f\|_{W^{1,\infty}}$ so that
\ben
\pa_z\rho_\delta(x,z)\ge \f {c_0} 2\quad \text{for}\quad (x,z)\in S.
\een
We denote
\beno
\Psi(x,z)\triangleq\Phi(x,\rho_\delta(x,z)),\quad \Delta=\Delta_x,\quad \na=\na_x.
\eeno
Then $\Psi(x,z)$ satisfies
\begin{equation}\label{eq:elliptic-flat}
\left\{\begin{array}{l}
\pa^2_z\Psi+\al \Delta\Psi+\beta\cdot\nabla \pa_z\Psi-\gamma \pa_z \Psi=0,\\
\Psi(x,0)=\psi,\quad \pa_z\Psi(x,-1)=0.
\end{array}\right.
\end{equation}
where the coefficients $\al,\beta,\gamma$ are given by
\beno
\alpha=\f{(\pa_z \rho_\delta)^2}{1+|\nabla \rho_\delta|^2},\quad \beta=-2\f{\pa_z\rho_\delta \nabla \rho_\delta}{1+|\nabla \rho_\delta|^2},
\quad\gamma=\f{1}{\pa_z\rho_\delta}(\pa_z^2\rho_\delta+\alpha\Delta\rho_\delta+\beta\cdot\nabla\pa_z\rho_\delta).
\eeno

We introduce the following functional spaces:
\beno
&&X^\sigma(I)\eqdefa  {L}^\infty_{z}(I;H^\sigma(\mathbf{T}^d))\cap L^2_{z}(I;H^{\sigma+\f12}(\mathbf{T}^d)),\\
&&Y^\sigma(I)\eqdefa  {L}^1_{z}(I;H^\sigma(\mathbf{T}^d))+L^2_{z}(I;H^{\sigma-\f12}(\mathbf{T}^d)).
\eeno

\begin{proposition}\label{prop:elliptic}
Let $\Phi\in H^1(S_f)$ be a weak solution of (\ref{eq:elliptic}).
Assume that $f\in H^{s+\f12}(\T^d)$ for $s>\f d 2+\f12$.
Then for any $\sigma\in [-\f12, s-\f12]$, it holds that
\begin{align}\label{eq:ellip-app}
&\|\nabla_{x,z}\Psi\|_{X^{\sigma}(I)}\leq C\big(c_0, \|f\|_{H^{s+\f12}}\big)\|\psi\|_{H^{\sigma+1}}.
\end{align}
\end{proposition}

To prove the proposition, we first paralinearize the elliptic equation (\ref{eq:elliptic-flat}) as
\ben\label{eq:elliptic-pl}
\pa^2_z\Psi+T_\al\Delta\Psi+T_\beta\cdot\nabla \pa_z\Psi=F_1+F_2,
\een
where $F_1, F_2$ are given by
\beno
F_1=\gamma \pa_z\Psi,\quad F_2=(T_\al-\al)\Delta\Psi+(T_\beta-\beta)\cdot\nabla \pa_z \Psi.
\eeno
Then we decouple the equation (\ref{eq:elliptic-pl}) into a forward and a backward parabolic
evolution equations:
\begin{eqnarray}\label{eq:elliptic-decouple}
(\pa_z-T_a)(\pa_z-T_A)\Psi=F_1+F_2+F_3\triangleq F,
\end{eqnarray}
where
\beno
&&a=\f12\big(-i\beta\cdot\xi-\sqrt{4\al|\xi|^2-(\beta\cdot\xi)^2}\big),\\
&&A=\f12\big(-i\beta\cdot\xi+\sqrt{4\al|\xi|^2-(\beta\cdot\xi)^2}\big),\\
&&F_3=(T_aT_A-T_\al\Delta)\Psi-(T_a+T_A+T_\beta\cdot\nabla )\pa_z\Psi-T_{\pa_z A}\Psi.
\eeno

We denote by $\Gamma_r^m(I\times \T^d)$ the space of symbols $a(z;x,\xi)$ satisfying
\beno
{\cM}^m_r(a)\eqdefa\sup_{z\in I}\sup_{|\alpha|\leq \f d 2+1+r}\sup_{|\xi|\geq1/2} \|(1+|\xi|)^{|\alpha|-m}\pa^\alpha_\xi a(z;\cdot,\xi)\|_{W^{r,\infty}}
<+\infty.
\eeno
It is easy to verify that if $f\in H^{s+\f12}(\T^d)$ for $s>\f d2+\f12$, then $a, A\in {\cM}^1_{\e}$ for some $\e>0$ with the bound
\beno
{\cM}^1_{\e}(a)+{\cM}^1_{\e}(A)\le C\big(c_0, \|f\|_{H^{s+\f12}}\big).
\eeno

The following lemma was basically proved in \cite{ABZ}.

\begin{lemma}\label{lem:parabolic}
Let $a\in \Gamma^1_\e(I\times \T^d)$ for some $\e>0$ be elliptic
in the sense that there exists $c_1>0$ such that for any $z\in I, (x,\xi)\in \T^2\times \R^2$,
\beno
\textrm{Re}\,a(z;x,\xi)\geq c_1|\xi|.
\eeno
Consider the parabolic evolution equation
\ben\label{eq:parabolic evolution}
\pa_z w+T_a w=g,\quad w|_{z=z_0}=w_0.
\een
If $w_0\in H^\sigma$ and $g\in Y^\sigma(I)$ for $\sigma\in \R$,
then there exists a unique solution $w\in X^\sigma(I)$ of (\ref{eq:parabolic evolution}) satisfying
\beno
\|w\|_{X^\sigma(I)} \leq C\big(c_1, {\cM}^1_{\e}(a)\big)\big( \|w_0\|_{H^\sigma}+\|g\|_{Y^\sigma(I)}\big).
\eeno
\end{lemma}

Fix $\e>0$ so that $\e<\max\big(\f12, s-\f d2-\f12\big)$. We have the following regularity estimates for $F_i, i=1,2,3$, which can be proved by using Lemma \ref{lem:bony}(see \cite{ABZ}).

\begin{lemma}\label{lem:F-est}
For any $\sigma\in [-\f12,s-\f12-\e]$, it holds that
\beno
&&\|F_1\|_{Y^{\sigma+\e}(I)}\le C\big(c_0, \|f\|_{H^{s+\f12}}\big)\|\pa_z\Psi\|_{X^\sigma(I)},\\
&&\|F_2\|_{Y^{\sigma+\e}(I)}\le C\big(c_0, \|f\|_{H^{s+\f12}}\big)\|\na_{x,z}\Psi\|_{X^\sigma(I)},\\
&&\|F_3\|_{L^2_z(I;H^{\sigma-\f12+\e})}\le C\big(c_0, \|f\|_{H^{s+\f12}}\big)\|\na_{x,z}\Psi\|_{L^2_z(I;H^{\sigma+\f12})}.
\eeno
\end{lemma}

Now we are in position to prove Proposition \ref{prop:elliptic}.
\begin{proof}
The proof uses the induction argument. First of all, it is easy to verify that the inequality is true for $\sigma=-\f12$.
Let us assume that it is also true for $\sigma\in \big(-\f12,s-\f12-\e\big]$. Thus, it suffices to show that
\begin{align}\label{eq:ellip-app-1}
&\|\nabla_{x,z}\Psi\|_{X^{\sigma+\e}(I)}\leq C\big(c_0, \|f\|_{H^{s+\f12}}\big)\|\psi\|_{H^{\sigma+1+\e}}\quad\text{for}\quad \sigma\in \big(-\f12,s-\f12-\e\big].\end{align}
Using a localization argument as in Lemma 2.8  in \cite{ABZ-D}, it can be proved that
\ben\label{eq:Psi-bot}
\|\Psi(x,-1)\|_{H^{s+\f12}}\le C(c_0)\|\psi\|_{H^\f12}.
\een
Set $W=(\pa_z-T_A)\Psi$, which satisfies
\beno
\pa_zW-T_aW=F,\quad W(-1)=-T_A\Psi|_{z=-1}.
\eeno
Using the induction assumption and (\ref{eq:Psi-bot}), it follows from Lemma \ref{lem:parabolic} and Lemma \ref{lem:F-est} that
\begin{align}\label{eq:W-est}
\|W\|_{X^{\sigma+\e}(I)}\leq& C\big(c_0, \|f\|_{H^{s+\f12}}\big)\Big(\|W(-1)\|_{H^{\sigma+\e}}+\|F\|_{Y^{\sigma+\e}(I)}\Big)\nonumber\\
\le& C\big(c_0, \|f\|_{H^{s+\f12}}\big)\Big(\|T_A\Psi|_{z=-1}\|_{H^{\sigma+\e}}+\|\nabla_{x,z}\Psi\|_{X^{\sigma}(I)}\Big)\nonumber\\
\le& C\big(c_0, \|f\|_{H^{s+\f12}}\big)\|\psi\|_{H^{\sigma+1}}.
\end{align}

We next consider the backward parabolic equation
\beno
\pa_z\Psi-T_A\Psi=W,\quad \Psi|_{z=0}=\psi.
\eeno
It follows from Lemma \ref{lem:parabolic} and (\ref{eq:W-est}) that
\begin{align}
\|\Psi\|_{X^{1+\sigma+\e}(I)}\leq& C\big(c_0, \|f\|_{H^{s+\f12}}\big)\big(\|\Psi\|_{H^{\sigma+1+\e}}+\|W\|_{L^2_z(I;H^{\sigma+\f12+\e})}\big)\nonumber\\
\le& C\big(c_0, \|f\|_{H^{s+\f12}}\big)\|\psi\|_{H^{\sigma+1+\e}}.\nonumber
\end{align}
Using $\pa_z\Psi=T_A\Psi+W$, we get by Proposition \ref{prop:symcal} that
\beno
\|\pa_z\Psi\|_{X^{\sigma+\e}(I)}\le C\big(c_0, \|f\|_{H^{s+\f12}}\big)\|\psi\|_{H^{\sigma+1+\e}}.
\eeno
This proves (\ref{eq:ellip-app-1}).
\end{proof}

\section*{Acknowledgment}
Wei Wang is partially supported by NSF of China under Grant 11501502.
Zhifei Zhang is partially supported by NSF of China under Grant
11371037 and 11425103.

\end{document}